\documentclass[11pt]{article}
\usepackage{amssymb,amsmath,setspace,graphicx,multicol,wrapfig,amsfonts,amsthm,titling,url,array,parskip,enumerate,mathtools,scrextend,bbm}
\usepackage{hyperref}
\usepackage{placeins}
\usepackage{aurical,pbsi}
\usepackage[T1]{fontenc}
\usepackage[hmargin=2.5cm,vmargin=2.5cm]{geometry}
\usepackage[format=plain, 
justification=raggedright,singlelinecheck=false]{caption}

\makeatletter
\def\thm@space@setup{%
  \thm@preskip=0.1in
  \thm@postskip=0in
}
\makeatother

\numberwithin{equation}{section}

\theoremstyle{plain}
\newtheorem{thm}{Theorem}[section]
\newtheorem{lem}[thm]{Lemma}
\newtheorem{prop}[thm]{Proposition}
\newtheorem{cor}[thm]{Corollary}

\theoremstyle{definition}
\newtheorem{defn}[thm]{Definition}
\newtheorem{conj}[thm]{Conjecture}
\newtheorem{rem}[thm]{Remark}

\theoremstyle{remark}

\DeclareMathOperator{\one}{\mathbbm{1}}
\DeclareMathOperator{\Prob}{Prob}

\DeclareMathOperator{\wt}{wt}
\DeclareMathOperator{\weight}{weight}

\DeclareMathOperator{\pr}{\sc{pr}}
\DeclareMathOperator{\fil}{fi}
\DeclareMathOperator{\com}{com}
\DeclareMathOperator{\bl}{bl}

\newcommand*\heavy{\includegraphics[width=0.1in]{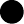}}
\newcommand*\light{\includegraphics[width=0.1in]{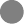}}

\begin{document}
\begin{center} {\Large{\sc Matrix Ansatz and combinatorics of the $k$-species PASEP}} \\
\vspace{0.1in}
Olya Mandelshtam
 \end{center}

\begin{abstract}
We study a generalization of the partially asymmetric exclusion process (PASEP) in which there are $k$ species of particles of varying weights hopping right and left on a one-dimensional lattice of $n$ sites with open boundaries. In this process, only the heaviest particle type can enter on the left of the lattice and exit from the right of the lattice. In the bulk, two adjacent particles of different weights can swap places. We prove a Matrix Ansatz for this model, in which different rates for the swaps are allowed. Based on this Matrix Ansatz, we define a combinatorial object which we call a $k$-rhombic alternative tableau, which we use to give formulas for the steady state probabilities of the states of this $k$-species PASEP. We also describe a Markov chain on the 2-rhombic alternative tableaux that projects to the 2-species PASEP.
\end{abstract}

\tableofcontents


\section{Introduction}
The (single-species) partially asymmetric exclusion process (PASEP) is an important non-equilibrium model that has generated interest in many areas of mathematics. Partly this is due to the existence of an exact solution (i.e.\ there are explicit formulas) for the stationary distribution for this process, which makes it a useful example in the study of non-equilibrium processes. Moreover, the PASEP has a rich combinatorial structure. Most notably, Corteel and Williams \cite{cw2007} gave a beautiful combinatorial interpretation of the stationary distribution of the PASEP in terms of certain tableaux called permutation tableaux, which are in bijection with several related objects, such as alternative tableaux, staircase tableaux, tree-like tableaux.

The single-species PASEP describes the dynamics of particles hopping on a finite one-dimensional lattice on $n$ sites with open boundaries, with the rule that there is at most one particle in a site, and at most one particle hops at a time. Figure \ref{PASEP_parameters} shows the parameters of this process, with the Greek letters denoting the rates of the hopping particles.

\begin{figure}
\centering
\includegraphics[width=0.35\textwidth]{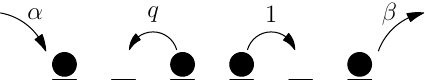}
\caption{The parameters of the PASEP.}
\noindent
\label{PASEP_parameters}
\end{figure}

More precisely, if we denote the particle by $d$ and the hole (or absence of a particle) by $e$, and let $X$ and $Y$ be any words in $\{e,d\}$ then the transitions of this process are:
\[ XdeY \overset{1}{\underset{q}{\rightleftharpoons}} XedY \qquad eX \overset{\alpha}{\rightharpoonup} dX \qquad \qquad Xd \overset{\beta}{\rightharpoonup} Xe\]
where by $X \overset{u}{\rightharpoonup} Y$ we mean that the transition from $X$ to $Y$ has probability $\frac{u}{n+1}$, $n$ being the length of $X$ (and also $Y$).

The $k$-species PASEP is a generalization of the single-species PASEP, where we now have $k$ particle types of varying weight hopping on a finite one-dimensional lattice of $n$ sites. We call the heaviest particle type $d$, the hole (or absence of any particles) type $e$, and the rest of the particle types from heaviest to lightest, type $a_1$ through type $a_{k-1}$. A heavier particle type can swap places with a lighter particle type with rate 1 if the heavier particle is on the left, and rate $q_i$ if the heavier particle is on the right, where $q_i$ is fixed to be some parameter out of a list of parameters $\{q_i\}$ (see Section \ref{k-species} for the precise rules). Like in the single-species PASEP, particles of type $d$ can enter on the left of the lattice and exit on the right of the lattice. Note that for $q_i=q$ for all $i$, when $k=1$, we recover the original single-species PASEP, and when $k=2$, we obtain the two-species PASEP, which has been studied, among others, by Uchiyama \cite{uchiyama}, Schaeffer \cite{schaeffer}, and Ayyer \cite{ayyer}. 

Derrida et. al. \cite{derrida} provided a Matrix Ansatz solution for the stationary distribution of the single-species PASEP. The Matrix Ansatz is a theorem that expresses the steady state probabilities of a process in terms of a certain matrix product for some matrices that satisfy some conditions that are determined by the process. Uchiyama provided a similar Matrix Ansatz solution, along with corresponding matrices, for the two-species PASEP. In this paper, we extend the result of Uchiyama to a Matrix Ansatz for a type of $k$-species PASEP in Theorem \ref{ansatz3}.

\begin{figure}[h]
\centering
\includegraphics[width=\linewidth]{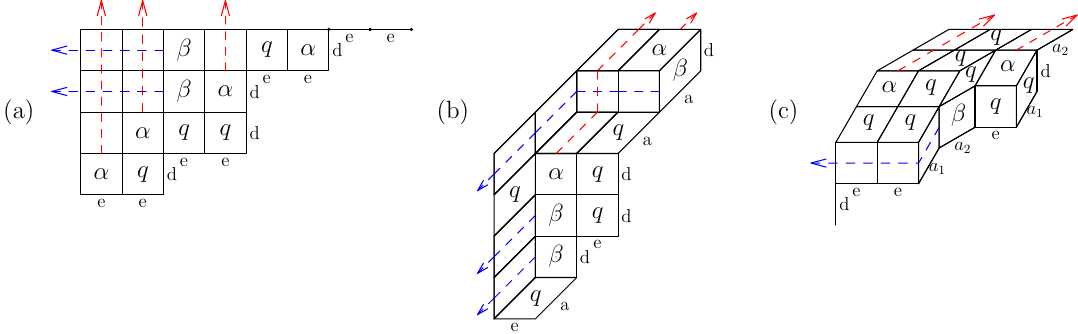}
\caption{(a) an alternative tableau, (b) a RAT, and (c) a 3-RAT}
\label{tableaux_samples}
\end{figure}

The tableaux of Corteel and Williams that provided a combinatorial interpretation for the stationary distribution of the single-species PASEP were generalized in \cite{omxgv} to certain tableaux called rhombic alternative tableaux (RAT) for the two-species PASEP. In this paper, we present even more general tableaux called the $k$-rhombic alternative tableaux ($k$-RAT), that give a combinatorial interpretation in Theorem \ref{kRAT_main} for the stationary distribution of the $k$-species PASEP. Figure \ref{tableaux_samples} shows examples of the alternative tableaux, the rhombic alternative tableaux, and the $k$-rhombic alternative tableaux. Note that 2-RAT are the same as RAT, and 1-RAT are the original alternative tableaux.

A common strategy to prove the combinatorial formula for the tableaux is to use the Matrix Ansatz, as in \cite{cw2007}. Another interesting and more direct proof of this formula is to explicitly construct a Markov chain on the tableaux, and show that it projects to the Markov chain of the underlying process, as in \cite{cw_mc}. As our final result, we construct such a Markov chain on the RAT that projects to the two-species PASEP.

Our paper is organized as follows. In Section \ref{2PASEP_sec}, we describe the two-species PASEP. In Section \ref{matrix_sec} of this paper, we provide a proof for Theorem \ref{prob} by explicitly defining the matrices that both provide the weight generating function of the RAT, and also satisfy the Matrix Ansatz hypothesis. 

In the second half of the paper, we describe a generalization of the two-species PASEP process to a $k$-species PASEP, and a generalization of the RAT to the $k$-RAT. Our proofs are analogous to the proofs for the two-species case. In Section \ref{k-species} of this paper, we introduce the $k$-species PASEP. In Section \ref{ansatz_sec}, we provide a Matrix Ansatz theorem that describes the stationary probabilities of this process as certain matrix products. In Section \ref{tableaux_sec}, we define the $k$-RAT, which provides an interpretation for the stationary probabilities of the $k$-species PASEP. Finally, in Section \ref{mc_sec}, we describe a Markov chain on the RAT that projects to the PASEP.

\textbf{Acknowledgement.} The author gratefully acknowledges Lauren Williams for her mentorship, and also Sylvie Corteel and Xavier Viennot for many fruitful conversations. The author was partially supported by the France-Berkeley Fund, the NSF grant DMS-1049513, and the NSF grant DMS-1704874.


\section{Previous results on the two-species PASEP}\label{2PASEP_sec}

First we describe the two-species PASEP and the associated rhombic alternative tableaux to motivate the more general $k$-species process that we study in this paper. 

The two-species partially asymmetric exclusion process (PASEP) has been studied extensively as an interesting generalization to the single-species PASEP. The two-species PASEP has two species of particles, one ``heavy'' and one ``light''. The ``heavy'' particle can enter the lattice on the left with rate $\alpha$, and exit the lattice on the right with rate $\beta$. Moreover, the ``heavy'' particle can swap places with both the hole and the ``light'' particle when they are adjacent, and the ``light'' particle can swap places with the hole. Each of these possible swaps occur at rate 1 when the heavier particle is to the left of the lighter one, and at rate $q$ when the heavier particle is to the right (we simplify our notation by treating the hole as a third type of ``particle''). The parameters of the two-species PASEP are shown in Figure \ref{2PASEP_parameters}, where the \heavy\ represents the ``heavy'' particle, and the \light\ represents the ``light'' particle.

\begin{figure}[h]
\centering
\includegraphics[width=0.6\linewidth]{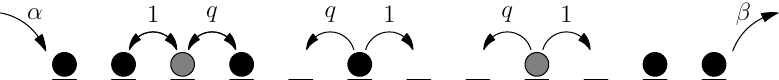}
\caption{The parameters of the two-species PASEP.}
\label{2PASEP_parameters}
\end{figure}

More precisely, if we denote the ``heavy'' particle by $d$, the ``light'' particle by $a$, and the hole by $e$, and let $X$ and $Y$ be any words in $\{d, a, e\}$ then the transitions of this process are:
\[  XdeY \overset{1}{\underset{q}{\rightleftharpoons}} XedY  \qquad  XaeY \overset{1}{\underset{q}{\rightleftharpoons}} XeaY \qquad XdaY \overset{1}{\underset{q}{\rightleftharpoons}} XadY\]

\[eX \overset{\alpha}{\rightharpoonup} dX \qquad \qquad Xd \overset{\beta}{\rightharpoonup} Xe\]
where by $X \overset{u}{\rightharpoonup} Y$ we mean that the transition from $X$ to $Y$ has probability $\frac{u}{n+1}$, $n$ being the length of $X$ (and also $Y$).

Notice that since only the ``heavy'' particle can enter or exit the lattice, the number of ``light'' particles must stay fixed. In particular, if we fix the number of ``light'' particles to be 0, we recover the original PASEP. 

Uchiyama provides a Matrix Ansatz along with matrices that satisfy the conditions, to express the stationary probabilities of the 2-species PASEP as a certain matrix product.

\begin{thm}[\cite{uchiyama}]\label{ansatz}
Let $W = W_1\ldots W_n$ with $W_i \in \{d, a, e\}$ for $1 \leq i \leq n$ represent a state of the two-species PASEP of length $n$ with $r$ ``light'' particles. Suppose there are matrices $D$, $E$, and $A$ and vectors $\langle w|$ and $|v \rangle$ which satisfy the following conditions
\[
DE= D+E+qED \qquad DA = A + qAD \qquad AE = A + qEA
\]
\[
\langle w| E = \frac{1}{\alpha} \langle w| \qquad D|v \rangle= \frac{1}{\beta} |v \rangle
\]
then 
\[
\Prob(W) = \frac{1}{Z_{n,r}} \langle w| \prod_{i=1}^n D\one_{(W_i=d)} + A\one_{(W_i=a)} + E\one_{(W_i=e)}|v \rangle
\]
where $Z_{n,r}$ is the coefficient of $y^r$ in $\frac{\langle w| (D + yA + E)^n|v \rangle}{\langle w | A^r|v \rangle}$.
\end{thm}

This result generalizes a previous Matrix Ansatz solution for the regular PASEP of Derrida et. al. in \cite{derrida}. 

In a previous paper, Mandelshtam and Viennot \cite{omxgv} introduced the rhombic alternative tableaux (RAT) which generalize the alternative tableaux and provide an explicit combinatorial formula for the stationary probabilities of the two-species PASEP. We describe the rhombic alternative tableaux below.

\begin{figure}[h]
\begin{minipage}{.48\textwidth}
  \centering
  \includegraphics[width=.52\linewidth]{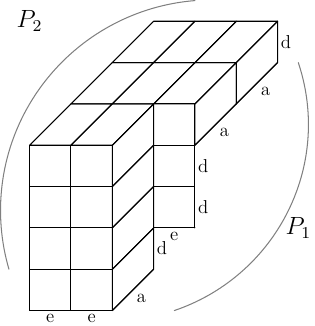}
  \caption{The rhombic diagram\\ $\Gamma(daaddedaee)$.}
  \label{maximal}
\end{minipage}%
\hfill
\begin{minipage}{.48\textwidth}
  \centering
  \includegraphics[width=.4\linewidth]{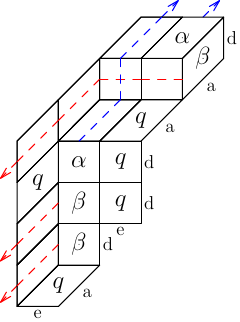}
  \caption{An example of a RAT of size $(9,3,4)$ with type $daaddedae$ and weight $\alpha^6\beta^5q^5$.}
  \label{RAT_example}
\end{minipage}
\end{figure}

\begin{defn}
Let $W$ be a word in the letters $\{d, a, e\}$ with $k$ $d$'s, $\ell$ $e$'s, and $r$ $a$'s of total length $n\vcentcolon= k+\ell+r$. Define $P_1$ to be the path obtained by reading $W$ from left to right and drawing a south edge for a $d$, a west edge for an $e$, and a southwest edge for an $a$. (From here on, we call any south edge a $d$-edge, any east edge an $e$-edge, and any southwest edge an $a$-edge.) Define $P_2$ to be the path obtained by drawing $\ell$ west edges followed by $r$ southwest edges, followed by $k$ south edges. A \textbf{rhombic diagram} $\Gamma(W)$ of \textbf{type} $W$ is a closed shape on the triangular lattice that is identified with the region obtained by joining the northeast and southwest endpoints of the paths $P_1$ and $P_2$ (see Figure \ref{maximal}).
\end{defn}

\begin{figure}[h]
\begin{minipage}{0.4\textwidth}
\centering
\includegraphics[width=0.6\textwidth]{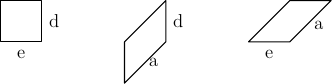}
\caption{The tiles $de$, $da$, and $ae$.}\label{tiles}
\end{minipage}\hfill
\begin{minipage}{0.6\textwidth}
\centering
\includegraphics[width=0.6\textwidth]{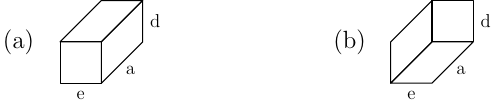}
\caption{(a) maximal and (b) minimal hexagons.}\label{hexagons}
\end{minipage}\hfill
\end{figure}

\begin{defn}
A \textbf{tiling} $\mathcal{T}$ of a rhombic diagram is a collection of open regions of the following three parallelogram shapes as seen in Figure \ref{tiles}, the closure of which covers the diagram:
\begin{itemize}
\item A parallelogram with south and west edges which we call a $de$ tile.
\item A parallelogram with southwest and west edges which we call an $ae$ tile.
\item A parallelogram with south and southwest edges which we call a $da$ tile.
\end{itemize}
We define the \textbf{area} of a tiling to be the total number of tiles it contains.
\end{defn}

\begin{defn}
The \textbf{size} of a RAT of type $W$ is $(n,k,r)$, where $k$ is the number of $d$'s in $W$, $r$ is the number of $a$'s in $W$, and $n$ is the total number of letters in $W$.
\end{defn}

\begin{defn}
A \textbf{$e$-strip} on a rhombic diagram with a tiling is a maximal strip of adjacent tiles of types $de$ or $ae$, where the edge of adjacency is always an $e$-edge. A \textbf{$d$-strip} is a maximal strip of adjacent tiles of types $de$ or $da$, where the edge of adjacency is always a $d$-edge. An \textbf{$a$-strip} is a maximal strip of adjacent tiles of types $da$ or $ae$, where the edge of adjacency is always an $a$-edge. 
\end{defn}

\begin{defn}
To compute the \textbf{weight} $\wt(F)$ of a filling $F$, first a $q$ is placed in every empty tile that does not have an $\alpha$ below it in the same $e$-strip or a $\beta$ to its right in the same $d$-strip. Next, $\wt(F)$ is the product of all the symbols inside $F$ times $\alpha^k \beta^{\ell}$, for $F$ a filling of size $(k+\ell+r,r,k)$.
\end{defn}

An example of a RAT is shown on the left of Figure \ref{RAT_example}.

It is easy to show with a weight-preserving bijection that we define in terms of certain ``flips'' on the tilings, that the sum of the weights of the tilings of $\Gamma(W)$ is independent of the tiling.\footnote{We define the flips more precisely in Section \ref{mc_sec}. The well-known fact that one can obtain any tiling of $\Gamma(W)$ from any other tiling with a series of flips is proved in \cite{omxgv}.}

\begin{prop}[\cite{omxgv} Proposition 2.8]\label{equiv}
Let $W$ be a word in $\{d, a, e\}$. Let $\mathcal{T}_1$ and $\mathcal{T}_2$ represent two different tilings of a rhombic diagram $\Gamma(W)$ with $de$, $da$, and $ae$ tiles. Then 
\[
\sum_{F \in \fil(W,\mathcal{T}_1)} \wt(F)=\sum_{F' \in \fil(W,\mathcal{T}_2)} \wt(F').
\]
\end{prop}

From the above, the definition below is well-defined.
\begin{defn}
Let $W$ be a word in $\{d, a, e\}$, and let $\mathcal{T}$ be an arbitrary tiling of $\Gamma(W)$.  Then the \textbf{weight} of a word $W$ is 
\[
\weight(W)=\sum_{F \in \fil(W,\mathcal{T})} \wt(F).
\]
\end{defn}

In \cite{omxgv}, it was shown that the rhombic alternative tableaux satisfy the same recursions as in the Matrix Ansatz, and therefore the weight generating function for the tableaux provides the stationary probabilities of the process. 

\begin{thm}[\cite{omxgv} Theorem 3.1]\label{prob}
Let $W$ be a word in $\{d, a, e\}^n$ that represents a state of the two-species PASEP with exactly $r$ $a$'s. Let 
\[\mathcal{Z}_{n,r} = \sum_{W'} \weight(W')\]
where $W'$ ranges over all words in $\{d, a, e\}^n$ with exactly $r$ $a$'s. Then the stationary probability of state $W$ is
\begin{equation}\label{eq_prob}
\Pr(W) = \frac{1}{\mathcal{Z}_{n,r}} \weight(W).
\end{equation}
\end{thm}


\section{Matrix Ansatz proof of Theorem \ref{prob}}\label{matrix_sec}

In this section we give a new proof of Theorem \ref{prob} by explicitly defining matrices $D$, $A$, and $E$ and row vector $\langle v|$ and column vector $|w \rangle$ that satisfy the hypotheses of a slightly more general Matrix Ansatz, and also have a combinatorial interpretation in terms of the RAT.

\subsection{Definition of our matrices}

Our matrices are infinite and indexed by a pair of non-negative integers in both row and column, so $D = [D_{(i,j)(u,v)}]_{i,j,u,v \geq 0}$, $A = [A_{(i,j)(u,v)}]_{i,j,u,v \geq 0}$, and $E = [E_{(i,j)(u,v)}]_{i,j,u,v \geq 0}$. Our vectors are also indexed by a pair of integers, so $\langle v| = [v_{(i,j)}]_{i,j \geq 0}$ and $|w \rangle = [w_{(u,v)}]^T_{u,v \geq 0}$.

We define $v_{(i,j)} = 1$ for $i=0, j=0$, and 0 for all other indices. We define $w_{(u,v)} = 1$ for all indices. 
\[
D_{(i,j)(i+1,j)} = \frac{1}{\beta} \]
and 0 for all other indices.

\[
A_{(i,j)(u,j+1)} = {i \choose u} q^u \beta^{i-u}
\]
for $0 \leq u \leq i$ and 0 for all other indices.

\[
E_{(i,j)(u,j)} = \frac{\beta^{i-u}}{\alpha} \left[ {i \choose u}q^u ( q^j+ \alpha [j]_q ) + \alpha \sum_{w=0}^{u-1} {i-u+w \choose i-u} q^w \right]
\]
for $0 \leq u \leq i$ and 0 for all other indices. (Here $[j]_q = q^{j-1}+\ldots+1$.)


Since $(i,j)$ specify the row of the matrices, and $(u,v)$ specify the columns, multiplication is defined as
\[
(M N)_{(i,j),(k,\ell)} = \sum_{u,v} M_{(i,j),(u,v)} N_{(u,v),(k,\ell)}.
\]
Note that in the case of the matrices $D$, $A$, and $E$, all products are given by finite sums, since the matrix entries are 0 for $u \geq i+1$ or $v \geq j+1$. 

To facilitate our proof, we provide a more flexible Matrix Ansatz that generalizes Theorem \ref{ansatz} with the same argument as in an analogous proof for the ordinary PASEP of Corteel and Williams \cite[Theorem 5.2]{cw2011}. For a word $W \in \{d, a, e\}^n$ with $r$ $a$'s, we define unnormalized weights $f(W)$ which satisfy
\[
\Pr(W)=f(W)/Z_{n,r}
\]
where $Z_{n,r}=\sum_{W'}f(W')$ where the sum is over all words $W'$ of length $n$ and with $r$ $a$'s.

\begin{thm}\label{ansatz2}
Let $\lambda$ be a constant. Let $\langle w|$ and $|v\rangle$ be row and column vectors with $\langle w| |v\rangle=1$. Let $D$, $E$, and $A$ be matrices such that for any words $X$ and $Y$ in $\{d, a, e\}$, the following conditions are satisfied:
\begin{enumerate}[(I)]
\item[I.] $\langle w| X(DE-qED)Y|v \rangle=\lambda \langle w| X(D+E)Y|v\rangle$,
\item[II.] $\langle w| X(DA-qAD)Y|v \rangle=\lambda \langle w| XAY|v\rangle$,
\item[III.] $\langle w| X(AE-qEA)Y|v \rangle=\lambda \langle w| XAY|v\rangle$,
\item[IV.] $\beta \langle w| XD |v \rangle=\lambda \langle w| X|v\rangle$,
\item[V.] $\alpha \langle w| EY |v \rangle=\lambda \langle w| Y|v\rangle$.
\end{enumerate} 
Let $W = W_1\ldots W_n$ with $W_i \in \{d, a, e\}$ for $1 \leq i \leq n$ represent a state of the two-species PASEP of length $n$ with $r$ ``light'' particles. Then
\[
f(W)= \frac{1}{\langle w| A^r |v \rangle} \langle w| \prod_{i=1}^n D\one_{(W_i=d)} + A\one_{(W_i=a)} + E\one_{(W_i=e)}|v \rangle.
\] 
\end{thm}

\begin{proof}
The proof of Theorem \ref{ansatz2} follows exactly that of \cite[Theorem 5.2]{cw2011}. Note that the above implies that 
\[
Z_{n,r}= [y^r]\frac{\langle w| (D+yA+E)^n | v \rangle}{\langle w| A^r |v \rangle}.
\]
\end{proof}

\subsection{Combinatorial interpretation of the matrices in terms of tableaux}\label{combinatorial_sec}

Let $W$ be an arbitrary word in $\{d,a,e\}$ with rhombic diagram $\Gamma(W)$ with the maximal tiling $\mathcal{T}_{max}$, and let $\weight(W)$ be the weight generating function for $\fil(W,\mathcal{T}_{max})$. Define a \textbf{free $d$-strip} to be a $d$-strip that does not contain a $\beta$. 
We call a $de$ or $da$ tile \textbf{free} if the $d$-strip adjacent to its east $d$-edge is a free $d$-strip. Note that any $de$ or $da$ tile that is not free must be empty.

\begin{defn} Let $W$ be a word in $\{d, a, e\}$. We define the \textbf{maximal tiling} of $\Gamma(W)$ to be the tiling that that does not contain an instance of a minimal hexagon (of Figure \ref{hexagons} (b)), for instance the tiling of the rhombic diagram in Figure \ref{maximal}. We refer to such a tiling by $\mathcal{T}_{max}$\footnote{$\mathcal{T}_{max}$ is the unique maximal tiling from \cite{omxgv}.}. $\mathcal{T}_{max}$ can be constructed by placing tiles from $P_1(W)$ inwards, and placing a $de$ tile with first priority whenever possible. In other words, all the north strips of $\mathcal{T}_{max}$ are, from bottom to top, a strip of adjacent $de$ boxes followed by a strip of adjacent $ae$ boxes. (A \textbf{minimal tiling} of $\Gamma(W)$ is correspondingly defined as the (unique) tiling that does not contain an instance of a maximal hexagon.)
\end{defn}

We will show that the previously defined matrices $D$, $A$, and $E$ represent the addition of a $d$-edge, an $a$-edge, and an $e$-edge to the bottom of $\Gamma(W)$ to form the rhombic diagram $\Gamma(Wd)$, $\Gamma(Wa)$, and $\Gamma(We)$ respectively. Recall that these matrices have rows indexed by the pair $(i,j)$ and columns indexed by the pair $(u, v)$. We let $i$ represent the number of free $d$-strips in a tableau $F \in \fil(W,\mathcal{T}_{max}(W))$, and $j$ the number of $a$'s in $W$. For the columns, we let $u$ represent the number of free $d$-strips in a tableau $F' \in \fil(Wd,\mathcal{T}_{max}(Wd))$ (and respectively, $\Gamma(Wa)$ and $\Gamma(We)$), and $v$ the number of $a$'s in $Wd$ (and respectively, $Wa$ and $We$). 

\begin{defn}
Let $X(W)$ be a word in $\{D, A, E\}$ representing a product involving the matrices $D$, $A$, and $E$, corresponding to the 2-PASEP word $W$ in the letters $\{d, a, e\}$. 
\end{defn}

For example, if $W=deaae$, then $X(W) = DEAAE$.

\begin{thm}\label{DAE_combinatorial}
Let $W$ be a word in $\{d, a, e\}$, and let $X=X(W)$. Then:
\begin{itemize}
\item $X_{(i,j)(u,v)}$ is the generating function for all ways of adding $|W|$ new edges of type $W$ to the southwest boundary of a rhombic alternative tableau with $i$ free $d$-strips and $j$ $a$-strips, to obtain a new rhombic alternative tableau with $u$ free $d$-strips and $v$ $a$-strips.
\item $(\langle w| X)_{(u,v)}$ is the generating function for rhombic alternative tableaux of type $W$, which have $u$ free $d$-strips and $v$ $a$-strips.
\item $\langle w| X | v \rangle$ is the generating function for all rhombic alternative tableaux of type $W$.
\end{itemize}
\end{thm}

\begin{figure}
\centering
\includegraphics[width=0.7\textwidth]{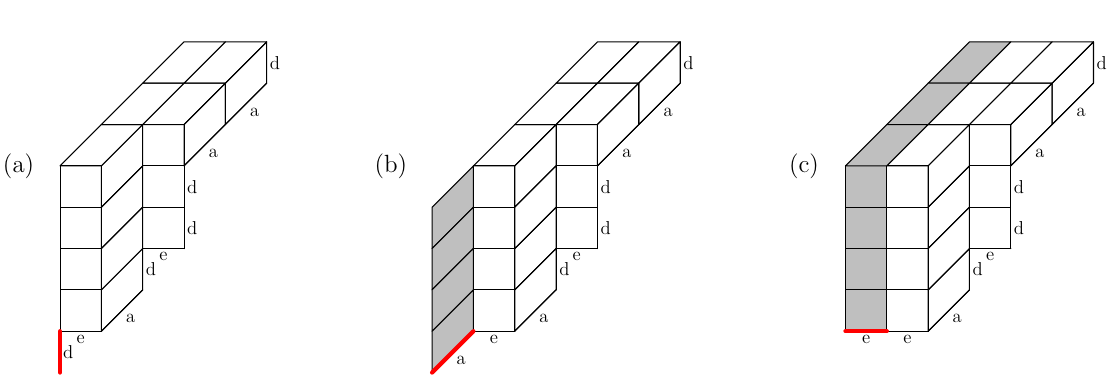}
\caption{Adding a (a) $d$, (b) $a$, or (c) $e$ to the end of $W$.}
\noindent
\label{Zn}
\end{figure}

We prove Theorem \ref{DAE_combinatorial} with the following lemma, which says that the matrices $D$, $A$, and $E$ are ``transfer matrices'' for building rhombic alternative tableaux with the maximal tiling.

\begin{lem}\label{transfer_matrices}
For the matrices $D$, $A$, and $E$,
\begin{itemize}
\item $D_{(i,j)(u,v)}$ is the generating function that represents the addition of a $d$-edge,
\item $A_{(i,j)(u,v)}$ is the generating function that represents the addition of an $a$-edge, and 
\item $E_{(i,j)(u,v)}$ is the generating function that represents the addition of an $e$-edge
\end{itemize}
to the southwest corner of a rhombic alternative tableau with the maximal tiling with $i$ free $d$-strips and $j$ $a$-strips, resulting in a rhombic alternative tableau with the maximal tiling with $u$ free $d$-strips and $j$ $a$-strips. 
\end{lem}

\begin{proof} We describe the possible rhombic alternative tableaux that arise from the addition of a $d$-edge, an $a$-edge, and an $e$-edge respectively to the southwest corner of an existing RAT of shape $W$ with the maximal tiling, and $i$ free $d$-strips and $j$ $a$-strips.

\noindent \textbf{The addition of the $d$-edge} to $\Gamma(W)$ does not affect the interior of the tableau, as in the example of Figure \ref{RAT_example} (a), and the tiling of the new tableau is clearly still a maximal one. Thus for any $F \in \fil(W, \mathcal{T}_{max}(W))$, we obtain $F' \in \fil(Wd, \mathcal{T}_{max}(Wd))$ whose weight simply increases by $\alpha$, the weight of the new $d$-edge. We have thus $\wt(F') = \alpha\wt(F)$. Moreover, the addition of the $d$-edge adds exactly one free $d$-strip to $F$. Thus we obtain the desired entry in the matrix $D$.

\noindent \textbf{The addition of the $a$-edge} and a vertical strip of adjacent $da$ tiles to the left boundary of $\mathcal{T}_{max}(W)$ results in a maximal tiling of $\Gamma(Wa)$, as in the example of Figure \ref{RAT_example} (b). Let us consider the entry $(i,j),(u,j+1)$ of $A$ for $0 \leq u \leq i$. Each free $da$ tile contains either a $q$ or a $\beta$ with no restrictions on their positions, for a total of $i-u$ $\beta$'s and $u$ $q$'s. Thus there are precisely ${i \choose u}$ ways to choose such a filling of the new tiles. Every such filling contributes a weight of $q^u \beta^{i-u}$. $Wa$ now has $j+1$ $a$'s, and it is clear that all other entries of $A$ are zero. Thus we obtain the desired entry in the matrix $A$.

\noindent \textbf{The addition of the $e$-edge} and a vertical strip of adjacent $de$ tiles followed by $j$ adjacent $ae$ tiles to the left boundary of $\mathcal{T}_{max}(W)$ results in a maximal tiling of $\Gamma(We)$, as in the example of Figure \ref{RAT_example} (c). Let us call this strip of new tiles the new $e$-strip. There are three possible cases for this new $e$-strip. For the following, let us consider the entry $(i,j),(u,j)$ of $E$ for $0 \leq u \leq i$. 

\emph{Case 1:} the new $e$-strip does not contain an $\alpha$. Then each of the $j$ $ae$ tiles must contain a $q$, and each of the $i$ free $de$ tiles contains either a $q$ or a $\beta$, with no restrictions on their positions, with exactly $i-u$ $\beta$'s and $u$ $q$'s. This gives a total weight contribution of ${i \choose u} \beta^{i-u} q^{u+j}$.

\emph{Case 2:} the new $e$-strip contains an $\alpha$ in one of the $ae$ tiles. Then each of the $ae$ tiles below that $\alpha$ must contain a $q$, and each of the free $i$ $de$ tiles contains either a $q$ or a $\beta$, with no restrictions on their positions, with exactly $i-u$ $\beta$'s and $u$ $q$'s. This gives a total weight contribution of $ {i \choose u} \alpha \beta^{i-u} q^{u} [j]_q$.

\emph{Case 3:} the new $e$-strip contains an $\alpha$ in one of the free $de$ tiles. Then exactly $i-u$ of the free $de$ tiles below the $\alpha$ must contain a $\beta$, and $u$ of them contain a $q$. This gives a total weight contribution of $\beta^{i-u} \alpha \sum_{w=0}^{u-1} {i-u+w \choose i-u} q^w$.

Thus we obtain the desired entry in the matrix $E$.
\end{proof}

\begin{proof}[Proof of \ref{DAE_combinatorial}]
The first point is immediate from Lemma \ref{transfer_matrices}.

The second point is due to the following: $\langle w|$ is a row vector for which the entry with index $(0,0)$ is 1, and the rest are 0. By the first point, $(\langle w| X)_{(0,0),(u,v)}$ is, in particular, the generating function for adding $|W|$ new edges of type $W$ to the southwest boundary of a trivial RAT of size 0, to result in a RAT of type $W$ with the maximal tiling with $u$ free $d$-strips and $v$ $a$-strips.

The third point is due to the following: $| v \rangle$ is a column vector with every entry equal to 1. By the second point, the generating function for all possible RAT in $\fil(W,\mathcal{T}_{max}(W))$ is the sum of RAT of type $W$ over all choices for the number of $a$-strips and free $d$-strips in the fillings. In other words, it is the sum over all $(u, v)$ of $(\langle w| X)_{(0,0),(u,v)}$. It follows that $\langle w| X | v \rangle$ is the desired generating function.
\end{proof}

\subsection{Combinatorial proof that our matrices satisfy the Matrix Ansatz}

Using Theorem \ref{DAE_combinatorial}, we provide simple combinatorial proofs that our matrices satisfy the equations of Theorem \ref{ansatz2}. Let $W$ be a word in $\{d, a, e\}$ with $\Gamma(W)$ its rhombic diagram. In this subsection, when we say ``addition of a $d$ (or $a$ or $e$) to $W$'', we mean adding a $d$-edge (or $a$- or $e$-edge) to the southwest point of $\Gamma(W)$, as described in the preceding subsection.

\textbf{I. $DE-qED=\alpha\beta(D+E)$}

By our construction, consecutive addition of a $d$ and an $e$ to $W$ results in a $de$ corner with a $de$ corner tile as the bottom-most tile of the $e$-strip that contains it (as well as the right-most tile of the $d$-strip that contains it). This $de$ corner tile contains an $\alpha$, $\beta$, or $q$. 
\begin{itemize}
\item If the $de$ corner tile contains an $\alpha$, then the rest of the $e$-strip containing this tile must be empty. Thus the entire $e$-strip has weight $\alpha\beta$, and the rest of the tableau has the same weight as if the $de$ were replaced by a $d$ (with the same filling in the corresponding tiles). 
\item If the $de$ corner tile contains a $\beta$, then the rest of the $d$-strip containing this tile must be empty. Thus the entire $d$-strip has weight $\alpha\beta$, and the rest of the tableau has the same weight as if the $de$ were replaced by an $e$ (with the same filling in the corresponding tiles).
\item If the $de$ corner tile contains a $q$, then this tile has no effect on the rest of the tableau which has the same weight as if the $de$ were replaced by an $ed$ (with the same tiling and filling), and the tile itself has weight $q$. 
\end{itemize}
Combining the above cases, we obtain that $DE = qED+ \alpha\beta(D+E)$, as desired.

%

\textbf{II. $DA-qAD=\alpha\beta A$}

By our construction, consecutive addition of a $d$ and an $a$ results in a $da$ corner with a $da$ corner tile as the right-most tile of the $d$-strip that contains it. This $de$ corner tile contains a $\beta$ or $q$. 
\begin{itemize}
\item If the $da$ corner tile contains a $\beta$, then the rest of the $d$-strip containing this tile must be empty. Thus the entire $d$-strip has weight $\alpha\beta$, and the rest of the tableau has the same weight as if the $da$ were replaced by an $a$ (with the same filling in the corresponding tiles).
\item If the $da$ corner tile contains a $q$, then this tile has no effect on the rest of the tableau which has the same weight as if the $da$ were replaced by an $ad$ (with the same tiling and filling), and the tile itself has weight $q$. 
\end{itemize}
Combining the above cases, we obtain that $DA = qAD+ \alpha\beta A$, as desired.

%

\textbf{III. $AE-qEA=\alpha\beta A$}

\begin{defn} We call an \textbf{$ae$ strip} the region of the rhombic diagram that corresponds to a maximal $a$-strip together with an adjacent maximal $e$-strip. (By maximal $a$- and $e$-strips, we mean $a$- and $e$-strips as they would appear in a maximal tiling of a rhombic diagram, i.e. a strip of adjacent $da$ tiles for the $a$-strip as in Figure \ref{Zn} (b), and a vertical strip of adjacent $de$ tiles followed by a strip of adjacent $ae$ tiles for the $e$-strip as in Figure \ref{Zn} (c).) We allow any valid tiling for the $ae$ strip, and we call an $ae$ strip \textbf{maximal} if it has the maximal tiling, and we call it \textbf{minimal} if it has the minimal tiling. Note that a minimal $ae$ strip has an $ae$ corner tile in the $ae$ corner.
\end{defn}

By our construction, consecutive addition of an $a$ and an $e$ results in a maximal $ae$ strip. For our proof, we consider the corresponding minimal $ae$ strip. We apply a series of flips to convert the maximal $ae$ strip to a minimal $ae$ strip, and we consider the contents of its $ae$ corner tile. This $ae$ corner can contain an $\alpha$ or $q$. If the $ae$ corner tile contains an $\alpha$, then the rest of the $e$-strip containing this tile must be empty. Thus the entire $e$-strip has weight $\alpha\beta$, and the rest of the tableau has the same weight as if the $e$-strip were removed entirely. This operation is the same as if in the original tableau, the $ae$ were replaced by an $a$ (with the same filling in the corresponding tiles). 

For the other case, if the $ae$ corner tile contains a $q$, then this tile has no effect on the rest of the tableau. Thus the weight of the tableau with the exception of the $ae$ corner tile is the same as the weight of a tableau with the same tiling and filling with the $ae$ replaced by an $ea$. Moreover, this new tableau (with the $ae$ corner tile removed from the minimal $ae$ strip) is in fact the maximal tableau that corresponds to replacing the $ae$ by an $ea$. Thus we have as desired, $AE = qEA + \alpha\beta A$ from these two cases.

\begin{rem}
It is also possible to directly compute the $(i,j),(u,v)$ entry of each term of the equations of Theorem \ref{ansatz2}, and show that equality holds in each case.
\end{rem}

\subsection{Properties of the RAT and enumeration}


\begin{wrapfigure}[7]{r}{0.4\textwidth}
\centering
\includegraphics[width=0.3\textwidth]{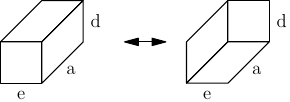}
\caption{A flip from a maximal (left) to a minimal hexagon (right).}
\noindent
\label{flip}
\end{wrapfigure}

\begin{defn} A \textbf{flip} is an involution that switches between a maximal hexagon and a minimal hexagon, and is the particular rotation of tiles that is shown in Figure \ref{flip}.
\end{defn}

It is a well-known result for rhombic tableaux that one can get from any tiling $\mathcal{T}$ to any other tiling $\mathcal{T}'$ with a series of flips. (This is elaborated upon in \cite{omxgv}.) In \cite[Definition 2.16]{omxgv}, we extend the flips on a tiling $\mathcal{T}$ of $\Gamma(W)$ for some word $W$ to weight-preserving flips on the RAT $T \in \fil(W,\mathcal{T})$, by explicitly defining the weight-preserving flip for each possible case of a filling of a hexagon in $\mathcal{T}$. This leads to the following definition. 

\begin{defn}\label{equivalent}
Let $W$ be a state of the two-species PASEP and $\mathcal{T}$ and $\mathcal{T}'$ be some tilings of $\Gamma(W)$. A RAT $F \in \fil(W,\mathcal{T})$ is \textbf{equivalent} to a RAT $F' \in \fil(W,\mathcal{T}')$ if $F$ can be obtained from $F'$ by some series of weight-preserving flips. 
\end{defn}

Let $\Omega^n_r$ be the set of states of the two-species PASEP of size $n$ with exactly $r$ ``light'' particles. Let $\Psi_{(n,r)}$ be the set of equivalence classes of RAT whose type belongs to $\Omega^n_r$. More precisely, $\psi \in \Psi_{(n,r)}$ is some set of RAT of a single type such that for any $F, F' \in \psi$, $F$ and $F'$ are equivalent. Moreover, if $F$ and $F'$ are equivalent and $F \in \psi$ and $F' \in \psi'$, then $\psi=\psi'$.

Finally, from \cite{omxgv}, we also have the following theorem.

\begin{thm}[\cite{omxgv} Theorem 2.19]
\[
\mathcal{Z}_{n,r}(\alpha,\beta,1) = {n \choose r} \prod_{i=r}^{n-1} (\alpha+\beta+i\alpha\beta).
\]
\end{thm}

This implies the following corollary.

\begin{cor}
\[
|\Psi_{(n,r)}| = {n \choose r} \frac{(n+1)!}{(r+1)!}.
\]
\end{cor}




\section{$k$-species PASEP}\label{k-species}

We now describe a generalization of the two-species PASEP to a $k$-species PASEP (also called $k$-PASEP) of a similar flavor. In our new model, we consider $k$ particle species of varying heaviness on a one-dimensional lattice of size $n$. We call the heaviest particle a $d$ particle, followed by $a_1> a_2 > \cdots > a_{k-1}$. For easier notation, we also introduce another particle which we call an $e$ particle to represent a hole, and we allow this to be the lightest particle in our set of species. Thus, in our model every location on the lattice contains exactly one particle, out of the set of species $\{d, a_1,\ldots,a_{k-1},e\}$. Moreover, the $d$ particle is allowed to ``enter'' on the left at location 1 by replacing an $e$ particle at that location (with rate $\alpha$), and it is allowed to ``exit'' on the right at location $n$ by being replaced with an $e$ particle at that location (with rate $\beta$). The particles of type $a_i$ are not allowed to enter or exit, so we fix the numbers of particles of those species to be $r_i$ for $i=1,\ldots,k-1$. 

For two particle types $A$ and $B$, we write $A>B$ (respectively, $A<B$ or $A=B$) to mean that $A$ is a heavier particle type than $B$ (respectively, $A$ is lighter than $B$, or they are equal). The dynamics in the bulk are the following: a heavier particle of species $A$ can swap with an adjacent lighter particle of species $B$ with rate 1 if $A$ is to the left of $B$, and with rate $0 \leq q_{AB} \leq 1$ if $A$ is to the right of $B$. This means that heavier particles have a tendency to move to the right of the lattice. Our notation is shown in the table below: 

\begin{center}
\begin{tabular}{c c r |c}
$A$& $B$& &$q_{AB}$\\
\hline
$d$ & $a_i$ &$1 \leq i \leq k-1$& $q_{0i}$ \\
$d$ & $e$ && $q_{0 \infty}$\\
$a_i$ & $a_j$ &$1 \leq j < i < k-1$& $q_{ij}$\\  
$a_i$ & $e$ & $1 \leq i \leq k-1$ & $q_{i\infty}$.\\
\end{tabular}
\end{center}

More precisely, our process is a Markov chain with states represented by words of length $n$ in the letters $\{d, a_1,\ldots,a_{k-1},e\}$. The transitions in the Markov chain are the following, with $X$ and $Y$ arbitrary words.
\[
Xa_i e Y \overset{1}{\underset{q_{i\infty}}{\rightleftharpoons}} Xea_i Y \qquad Xde Y \overset{1}{\underset{q_{0\infty}}{\rightleftharpoons}} Xed Y \qquad Xda_i Y \overset{1}{\underset{q_{0 i}}{\rightleftharpoons}} Xa_i d Y \qquad Xa_i a_j Y \overset{1}{\underset{q_{i j}}{\rightleftharpoons}} Xa_ja_iY
\]
\[
e X \overset{\alpha}{\rightharpoonup} d X \qquad \qquad Xd \overset{\beta}{\rightharpoonup} Xe
\]
for $1 \leq i \leq k-1$ and $1 \leq j<i$.

where by $X \overset{u}{\rightharpoonup} Y$ we mean that the transition from $X$ to $Y$ has probability $\frac{u}{n+1}$, $n$ being the length of $X$ (and also $Y$).

\begin{defn}
For a given $k$-PASEP, we fix $n$ to be the size of the lattice and $r_i$ to be the number of particles of species $a_i$ for $1 \leq i \leq k-1$. We define $\Omega^n_{r_1,\ldots,r_{k-1}}$ to be the set of words of length $n$ in the letters $\{d, a_1, \ldots, a_{k-1}, e\}$ with $r_i$ instances of the letter $a_i$ for each $i$. We also define 
\[
\Omega^n = \bigcup_{r_1,\ldots,r_{k-1}} \Omega^n_{r_1,\ldots,r_{k-1}}.
\]
\end{defn}

\begin{rem}
In Section \ref{ansatz_sec}, we will provide a Matrix Ansatz solution for the model with different parameters $q_i$ for every type of transition. However, so far we only have nice combinatorics when all the $q_i$'s are set to equal a single constant $q$. Furthermore, it is easy to see that if $k=2$, we recover the 2-species PASEP that we described in the previous section, and if $k=1$, we recover the original PASEP. 
\end{rem}


\section{Matrix Ansatz solution for the $k$-PASEP}\label{ansatz_sec}

Building on a Matrix Ansatz solution for the usual PASEP by Derrida at. al. \cite{derrida} and a more general solution for the two-species PASEP by Uchiyama in \cite{uchiyama}, we have the following generalization for the $k$-species process.

\begin{thm}\label{ansatz3}
Let $W = W_1\ldots W_n$ with $W_i \in \{d, a_1,\ldots,a_{k-1}, e\}$ for $1 \leq i \leq n$ represent a state of the $k$-species PASEP in $\Omega^n_{r_1,\ldots,r_{k-1}}$. Suppose there are matrices $D$, $A_1,\ldots,A_{k-1}$, and $E$ and a row vector $\langle w|$ and a column vector $|v \rangle$ (with $\langle w| |v \rangle = 1$) which satisfy the following conditions
\begin{equation}\label{ansatz_bulk}
DE-q_{0\infty}ED = D+E, \quad DA_i-q_{0 i} A_iD = A_i,  \quad A_i E-q_{i \infty}E A_i = A_i, \quad A_i A_j - q_{i j} A_j A_i = 0,
\end{equation}
\begin{equation}\label{ansatz_boundary}
\langle w| E = \frac{1}{\alpha} \langle w|, \qquad D|v \rangle= \frac{1}{\beta} |v \rangle,
\end{equation}
then 
\[
\Prob(W) = \frac{1}{Z_{n,r_1,\ldots,r_{k-1}}} \langle w| \prod_{i=1}^n D \one_{(W_i=d)} + E \one_{(W_i=e)} + \sum_{i=1}^{k-1}A_i \one_{(W_i=a_i)} |v \rangle
\]
where $Z_{n,r_1,\ldots,r_{k-1}}$ is the coefficient of  $y_1^{r_1}\ldots y_{k-1}^{r_{k-1}}$ in 
\[
\frac{\langle w| (D + y_1A_1+\cdots+y_{k-1}A_{k-1} + E)^n|v \rangle}{\langle w | A_{k-1}^{r_{k-1}}\cdots A_1^{r_1}|v \rangle}.
\]
\end{thm}

\begin{proof}
For $W$ a word of length $n$, we define the weight 
\[f_n(W)= \langle w| \prod_{i=1}^n D \one_{(W_i=d)} + E \one_{(W_i=e)} + \sum_{i=1}^{k-1}A_i \one_{(W_i=a_i)} |v \rangle.\]

We show that $f_n(W)$ satisfies the detailed balance conditions
\begin{equation}\label{balance}
f_n(W) \sum_{W \rightarrow V} \Pr(W \rightarrow V) = \sum_{X \rightarrow W} f_n(X) \Pr(X \rightarrow W)
\end{equation}
for each $W \in \Omega^n$, where by $\Pr(W \rightarrow V)$ and $\Pr(X \rightarrow W)$ we denote the probabilities of the transitions $W \rightarrow V$ and $X \rightarrow W$ respectively. This would imply that the stationary probability of state $W$ is proportional to $f_n(W)$, which would complete the proof.

We observe that for fixed $W$, the only terms $f_n(X) \Pr(X \rightarrow V)$ for some $X,V \in \Omega^n$ appearing in \eqref{balance}, are precisely the terms:
\begin{itemize}
\item $f_n(e W_2\ldots W_n)\alpha$, 
\item $f_n(W_1\ldots W_{n-1}d)\beta\}$,
\item and $\{f_n(W_1\ldots W_{i-1}B C W_{i+2}\ldots W_n)\cdot 1, -f_n(W_1\ldots W_{i-1}C B W_{i+2}\ldots W_n) \cdot q_{BC}\}$ where $W_i W_{i+1}=BC$ for $B>C$ over $1 \leq i \leq n-1$.
\end{itemize}
This is because these terms are precisely the terms out of which possible transitions can occur to go into or out of $W$. Moreover, whether those terms appear on the left hand side of Equation \eqref{balance} or the right hand side is determined by whether $W_i W_{i+1}=BC$ or $W_i W_{i+1}=CB$ for $B>C$. In other words, the terms in the bulk are given a sign of $(-1)^{\one_{(W_{i+1}>W_i)}}$ for each $i$, and the boundary terms are given a sign of $(-1)^{\one_{(W_1=d)}}$ and $(-1)^{\one_{(W_n=e)}}$ for the left and right boundaries, respectively.

Thus Equation \eqref{balance} can be rewritten as the following:
\begin{multline}\label{balance2}
\one_{(W_1 =d\ \mbox{or}\ e)}(-1)^{\one_{(W_1=d)}} \alpha f_n(eW_2\ldots W_n)\\
+ \one_{(W_n =d\ \mbox{or}\ e)}(-1)^{\one_{(W_n=e)}} \beta f_n(W_1\ldots W_{n-1}d)\\
+ \sum_{i=1}^{n-1} \one_{(W_i \neq W_{i+1})}(-1)^{\one_{(W_{i+1}>W_i)}} \Big( f_n(W_1\ldots W_{i-1}B_i C_i W_{i+2}\ldots W_n) \\
 - q_{B_iC_i}f_n(W_1\ldots W_{i-1}C_i B_i W_{i+2}\ldots W_n)\Big)\\
\end{multline} 
where in the above we use $B_i \vcentcolon= \max(W_i,W_{i+1})$ and $C_i \vcentcolon= \min(W_i,W_{i+1})$.

The reduction rules of Equation \eqref{ansatz_bulk} or \eqref{ansatz_boundary} apply whenever $W_1 = d\ \mbox{or}\ e$, or $W_n = d\ \mbox{or}\ e$, or whenever $W_i\neq W_{i+1}$ for $1 \leq i <n$. We obtain the following.
\begin{align}\label{f_n}
f_n(W' de W'') - q_{0\infty}f_n(W' ed W'') &= f_{n-1}(W' d W'') + f_{n-1}(W' e W''),\\
f_n(W' da_i W'') - q_{0 i}f_n(W' a_id W'') &= f_{n-1}(W' a_i W''),\\
f_n(W' a_ie W'') - q_{i\infty}f_n(W' ea_i W'') &= f_{n-1}(W' a_i W''),\\
f_n(W' a_ia_j W'') - q_{i j}f_n(W' a_ja_i W'') &= 0,\\
\alpha f_n(e W'') &= f_{n-1}(W''),\\
\beta f_n(W' d) & = f_{n-1}(W').
\end{align}

For $W=W_1\ldots W_n$, we introduce the notation $f_{n-1}^i (W) = f_{n-1}(W_1\ldots\hat{W_i}\ldots W_n)$ to be the weight of the word $W$ with the letter $W_i$ cut out. With this notation, using the reduction rules of Equation \eqref{f_n}, Equation \eqref{balance2} becomes the sum $a_0+a_1+\ldots+a_{n-1}+a_n$, where
\[
a_0= \begin{cases}
f_{n-1}^1(W)& W_1=e\\
-f_{n-1}^1(W)&W_1=d
\end{cases}
,\qquad
a_n= \begin{cases}
f_{n-1}^n(W)& W_n=d\\
-f_{n-1}^n(W)&W_n=e
\end{cases},
\]
\begin{equation}\label{total}
\mbox{and} \quad a_i=\begin{cases}
f_{n-1}^i(W) + f_{n-1}^{i+1}(W) & \mbox{if } W_i W_{i+1} = de\ \mbox{or}\ ed\\
f_{n-1}^i(W) & \mbox{if } W_i W_{i+1} = da_i\\
-f_{n-1}^{i+1}(W) & \mbox{if } W_i W_{i+1} = a_id\\
f_{n-1}^{i+1}(W) & \mbox{if } W_i W_{i+1} = a_i e\\
-f_{n-1}^i(W) & \mbox{if } W_i W_{i+1} = e a_i\\
\end{cases} \quad \mbox{for } 1 \leq i \leq n-1.
\end{equation}

Notice that for all $i>j$, the terms $f_n(W'a_ia_j W'')-q_{i,j} f_n(W'a_ja_i W'')=0$.

Suppose there are a total of $s$ transitions in the bulk. For $j=1,\ldots,s$, label the location $i$ where the $j$'th transition occurs (i.e.\ the $j$'th $i$ for which $W_i \neq W_{i+1}$) by $W_{t_j}$. The strategy of our proof is to show that all the $f_{n-1}$ terms that arise from the transitions at the locations $\{t_j\}_{1 \leq j \leq s}$ cancel with other terms Equation \eqref{total} with an opposite sign. We describe these cancellations in the cases that follow.

\begin{enumerate}[(a.)]
\item $W_{t_j}W_{t_j+1}=de$, so the contribution of terms from this transition is $f_{n-1}^{t_{j}}(W)+f_{n-1}^{t_{j}+1}(W)$. Then $W_{t_{j+1}}W_{t_{j+1}+1}$ is necessarily either $ed$ or $ea_t$ for some $t$, in which case it contributes the term $-f_{n-1}^{t_{j+1}}(W)$. Similarly, $W_{t_{j-1}}W_{t_{j-1}+1}$ is necessarily either $de$ or $a_ue$ for some $u$, in which case it contributes the term $-f_{n-1}^{t_{j-1}+1}(W)$. However, the former of these cancels with the term $f_{n-1}^{t_j}(W)$, and the latter cancels with $f_{n-1}^{t_j+1}(W)$, as desired. 

There are two exceptions to the above. First, if $j=1$, then there is no $t_{j-1}$ term. However, in this case, $W$ necessarily begins with a $d$, and so the $f_{n-1}^{t_j}(W)$ term cancels with the left boundary term $-f_{n-1}^1(W)$. Second, if $j=n$, then there is no $t_{j+1}$ term. However, in this case, $W$ necessarily ends with an $e$, and so the $f_{n-1}^{t_j+1}(W)$ term cancels with the right boundary term $-f_{n-1}^n(W)$.

\item $W_{t_j}W_{t_j+1}=ed$, so the contribution of terms from this transition is $-f_{n-1}^{t_{j}}(W)-f_{n-1}^{t_{j}+1}(W)$. Then $W_{t_{j-1}}W_{t_{j-1}+1}$ is necessarily either $de$ or $a_te$ for some $t$, in which case it contributes the term $f_{n-1}^{t_{j-1}+1}(W)$. Similarly, $W_{t_{j+1}}W_{t_{j+1}+1}$ is necessarily either $de$ or $da_u$ for some $u$, in which case it contributes the term $f_{n-1}^{t_{j+1}}(W)$. However, the former of these cancels with the term $-f_{n-1}^{t_j}(W)$, and the latter cancels with $-f_{n-1}^{t_j+1}(W)$, as desired. 

There are two exceptions to the above. First, if $j=1$, then there is no $t_{j-1}$ term. However, in this case, $W$ necessarily begins with an $e$, and so the $-f_{n-1}^{t_j}(W)$ term cancels with the left boundary term $f_{n-1}^1(W)$. Second, if $j=n$, then there is no $t_{j+1}$ term. However, in this case, $W$ necessarily ends with a $d$, and so the $-f_{n-1}^{t_j+1}(W)$ term cancels with the right boundary term $f_{n-1}^n(W)$.

The rest of the cases are similar. Below, we describe the cancellations that occur for each transition location.

\item $W_{t_j}W_{t_j+1}=da_t$, so the contribution of terms from this transition is $f_{n-1}^{t_j}(W)$. This term cancels with the term $-f_{n-1}^{t_{j-1}+1}(W)$ since $W_{t_{j-1}}W_{t_{j-1}+1}$ must equal $ed$ or $a_ud$ for some $u$.

\item $W_{t_j}W_{t_j+1}=a_td$, so the contribution of terms from this transition is $-f_{n-1}^{t_j+1}(W)$. This term cancels with the term $f_{n-1}^{t_{j+1}}(W)$ since $W_{t_{j+1}}W_{t_{j+1}+1}$ must equal $de$ or $da_u$ for some $u$.

\item $W_{t_j}W_{t_j+1}=a_te$, so the contribution of terms from this transition is $f_{n-1}^{t_j+1}(W)$. This term cancels with the term $-f_{n-1}^{t_{j+1}}(W)$ since $W_{t_{j+1}}W_{t_{j+1}+1}$ must equal $ed$ or $ea_u$ for some $u$.

\item $W_{t_j}W_{t_j+1}=ea_t$, so the contribution of terms from this transition is $-f_{n-1}^{t_j}(W)$. This term cancels with the term $f_{n-1}^{t_{j-1}+1}(W)$ since $W_{t_{j-1}}W_{t_{j-1}+1}$ must equal $de$ or $a_ue$ for some $u$.

\end{enumerate}

The cancellations of the boundary terms are treated as the exceptions in cases of (a) and (b).

It is easy to check from the above that every term cancels with another term in Equation \eqref{balance2}, so indeed, it equals zero. Thus the function $f_n$ satisfies the detailed balance in Equation \eqref{balance}, as desired.
\end{proof}


\section{$k$-rhombic alternative tableaux}\label{tableaux_sec}

In this section, we introduce a combinatorial object that generalizes the RAT to provide an interpretation for the probabilities of the $k$-PASEP. This object, called the $k$-rhombic alternative tableau (or $k$-RAT) is of the same flavor as the RAT, and is similarly defined as follows.

\subsection{Definition of $k$-RAT}

To a word $W \in \Omega^n_{r_1,\ldots,r_{k-1}}$, we associate a $k$-rhombic diagram $\Gamma(W)$ as follows. 

\begin{defn} Let $W \in \Omega^n_{r_1,\ldots,r_{k-1}}$, and let $r_0$ be the number of $e$'s and $r_k$ the number of $d$'s in $W$. Let an $e$-edge be a unit edge oriented in the direction $-\pi$. Let a $d$-edge be a unit edge oriented in the direction $-\pi/2$. Let an $a_i$-edge be a unit edge oriented in the direction $-\frac{(k+i)\pi}{2k}$ (see Figure \ref{kRAT_edges}). Define $P_1(W)$ to be the lattice path composed of the $e$-, $a_1$-, \ldots, $a_{k-1}$-, and $d$-edges, placed end to end in the order the corresponding letters appear in the word $W$. Define $P_2(W)$ to be the path obtained by placing in the following order: $r_0$ $e$-edges, $r_1$ $a_1$-edges, $r_2$ $a_2$ edges, and so on, up to $r_{k-1}$ $a_{k-1}$-edges, and then $r_k$ $d$-edges. The \textbf{$k$-rhombic diagram} $\Gamma(W)$ is the closed shape that is identified with the region obtained by joining the northwest and southwest endpoints of $P_1(W)$ and $P_2(W)$ (see Figure \ref{kRAT_maximal}). 

Define a lattice path given by $W$ to be composed of the edges in the order they appear in the word $X$, and let us associate this lattice path with the southeast boundary of our rhombic diagram. We complete the path to form the diagram by drawing in the following order: to connect the top-most corner of the lattice path to its bottom-most corner. 
\end{defn}

\begin{figure}[h]
\begin{minipage}{0.48\textwidth}
\centering
\includegraphics[width=0.45\textwidth]{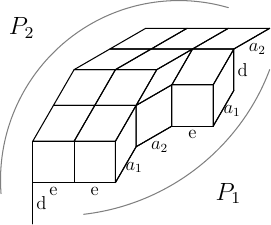}
\caption{$\Gamma(a_2da_1ea_2a_1eed)$ defined by $P_1$ and $P_2$, with a maximal tiling.}
\noindent
\label{kRAT_maximal}
\end{minipage}
\hfill
\begin{minipage}{0.48\textwidth}
\centering
\includegraphics[width=0.45\textwidth]{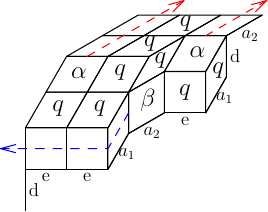}
\caption{A $3$-RAT of type $a_2da_1ea_2a_1eed$ of weight $\alpha^4\beta^4q^8$.}
\noindent
\label{kRAT_example}
\end{minipage}
\end{figure}


\begin{defn} A $de$ tile is a rhombus with $d$ and $e$ edges. A $da_i$ tile is a rhombus with $d$ and $a_i$ edges. An $a_ie$ tile is a rhombus with $a_i$ and $e$ edges. An $a_ia_j$ tile is a rhombus with $a_i$ and $a_j$ edges for $i>j$ (see Figure \ref{kRAT_tiles}). We impose on the tiles the following partial ordering: $a_j X<a_i X'<d X''$, and $Xe<Xa_j<Xa_i<Xd$ for $i>j$ and for any edges $X, X', X''$. If tile C < tile D according to our ordering, we say D is \textbf{heavier} than C.
\end{defn}

\begin{figure}[h]
\begin{minipage}{0.48\textwidth}
\centering
\includegraphics[width=0.8\textwidth]{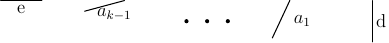}
\caption{$e$-edge, $a_{k-1}$-edge, $\ldots$, $a_{1}$-edge, $d$-edge}
\noindent
\label{kRAT_edges}
\end{minipage}
\hfill
\begin{minipage}{0.48\textwidth}
\centering
\includegraphics[width=0.8\textwidth]{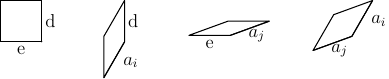}
\caption{A $de$-tile, $da_i$-tile, $a_i e$-tile, and $a_ia_j$-tile (with $i<j$)}
\noindent
\label{kRAT_tiles}
\end{minipage}
\end{figure}

\begin{defn} A \textbf{maximal} tiling on a $k$-rhombic diagram is one in which tiles are always placed from southeast to northwest, and priority is always given to the ``heaviest'' tiles.
\end{defn}


Define a \emph{maximal corner} to be a corner on $P_1(W)$ whose edges $A$ and $B$ are such that for any other corner on that diagram with edges $C$ and $D$, $AB \geq CD$. The canonical way to tile the rhombic diagram with a maximal tiling would be to pick a maximal corner with some edges $A$ and $B$, and place an $AB$ tile adjacent to that corner. The rest of the surface would then itself be a rhombic diagram with the same $P_2$.  We proceed to tile that surface in the same manner until the untiled region has area zero. It is easy to see that such a construction results in a maximal rhombic tiling of the $k$-rhombic diagram. Let us call this tiling $\mathcal{T}(W)$.

\begin{defn} 
An \textbf{$e$-strip} is a maximal strip of adjacent tiles whose edge of adjacency is an $e$-edge, as in Figure \ref{kRAT_strips} (a). A \textbf{$d$-strip} is a maximal strip of adjacent tiles whose edge of adjacency is a $d$-edge, as in Figure \ref{kRAT_strips} (b). (This definition is the same for the $k$-RAT as it is for the RAT).
\end{defn}

We now define a filling of $\mathcal{T}(W)$ with $\alpha$'s and $\beta$'s as follows.


\begin{figure}
\centering
\includegraphics[width=0.7\textwidth]{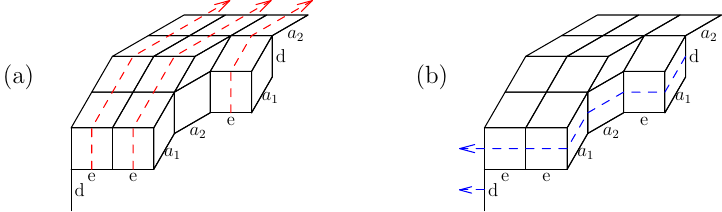}
\caption{(a) $e$-strips and (b) $d$-strips.}
\noindent
\label{kRAT_strips}
\end{figure}

\begin{defn}\label{kRAT_def}
A filling of a $k$-RAT is defined by the following rules.
\begin{itemize}
\item A $de$-tile is allowed to be empty or contain $\alpha$ or $\beta$.
\item A $da_i$ tile is allowed to be empty or contain $\beta$, for each $i$.
\item An $a_i e$ tile is allowed to be empty or contain $\alpha$, for each $i$.
\item An $a_i a_j$ tile must be empty, for each $i>j$.
\item Any tile in the same $e$-strip and above an $\alpha$ must be empty.
\item Any tile in the same $d$-strip and left of a $\beta$ must be empty.
\end{itemize}
\end{defn}


Denote the set of fillings of $\mathcal{T}(W)$ by $\fil(W)$. We assign weights to a filling $F \in \fil(W)$ from the rules above by placing a $q$ in each tile that is not forced to be empty by some $\alpha$ below it in the same $e$-strip, or some $\beta$ to the right in the same $d$-strip. For an example, see Figure \ref{kRAT_example}.\footnote{We allow the parameters $q_{BC}$ that represent swapping rates between B-type and C-type particles to vary in Section \ref{ansatz_sec}. However, to keep the combinatorics ``nice'', we fix all these parameters to equal a single constant $q$.}

\begin{defn} 
Let $W \in \Omega^n$, and $t$ be the number of $d$'s and $\ell$ the number of $e$'s in $W$. For $F \in \fil(W)$, define the \textbf{weight} $\wt(F)$ to be the product of the symbols in the filling of $F$ times $\alpha^t \beta^{\ell}$.
\end{defn}

Define 
\[
\mathcal{Z}_{n,r_1,\ldots,r_{k-1}} = \sum_{W} \sum_{F \in \fil(W)} \wt(F)
\] 
to be the sum of the weights over all $k$-RAT corresponding to states in $\Omega^n_{r_1,\ldots,r_{k-1}}$. Our main result for the $k$-RAT is the following, which we will prove in the next section.

\begin{thm}\label{kRAT_main}
Let $\fil(W)$ denote the set of fillings of the rhombic diagram $\Gamma(W)$ with the maximal tiling, and let $\wt(F)$ denote the weight of a filling in $\fil(W)$.Then the stationary probability of state $W$ of the $k$-PASEP is
\[
\frac{1}{\mathcal{Z}_{n,r_1,\ldots,r_{k-1}}} \sum_{F\in \fil(W)} \wt(F).
\]
\end{thm}

\begin{conj}
Let $\mathcal{T'}$ be any tiling of the rhombic diagram $\Gamma(W)$ associated to a state $W$ of the $k$-PASEP. Let $\fil(W,\mathcal{T}')$ denote the set of fillings of tiling $\mathcal{T}'$. Then the stationary probability of state $W$ of the $k$-PASEP is
\[
\frac{1}{\mathcal{Z}_{n,r_1,\ldots,r_{k-1}}} \sum_{F\in \fil(W,\mathcal{T}')} \wt(F).
\]
\end{conj}

\begin{rem} When $k=2$, the conjecture above is true, as we can define a weight-preserving bijection between fillings of the arbitrary tiling $\mathcal{T'}$ and fillings of the maximal tiling $\mathcal{T}$, in terms of ``flips'' \cite{omxgv}. For $k\geq 3$, flips admit the same weight-preserving bijection, but it is no longer necessarily the case that we could obtain any tiling $\mathcal{T'}$ of a $k$-rhombic diagram via flips from the maximal tiling.
\end{rem}

\subsection{Matrix Ansatz proof for the $k$-RAT}

We will prove Theorem \ref{kRAT_main} using the same strategy as in Section \ref{matrix_sec} for the RAT.

We provide matrices $D, E, A_1,\ldots,A_{k-1}$ that correspond to the addition of a $d$-edge, $e$-edge, or $a_i$-edge for $1 \leq i \leq k-1$ to the bottom of the path corresponding to a word $W$ of length $n$ to form a new rhombic diagram with a maximal tiling of size $n+1$ that corresponds to the word $Wd$ (or $We$, or $Wa_i$ for $1 \leq i \leq k-1$ respectively). For $\lambda=\alpha\beta$, we show that these matrices satisfy the Matrix Ansatz relations 
\begin{equation}\label{k-ansatz}
DE-qED=\lambda(D+E), \qquad DA_i - qA_iD = \lambda A_i, \qquad A_iE-qEA_i = \lambda A_i, \qquad A_iA_j=qA_jA_i \mbox{ for } i>j.
\end{equation}
The $k$-species Matrix Ansatz of Theorem \ref{ansatz3} would then imply that the steady state probability of $k$-PASEP state $W$ is proportional to a certain matrix product $\langle w| X(W) |v \rangle$ with the matrices $\{D, E, A_1,\ldots, A_{k-1}\}$. (As in Section \ref{ansatz_sec}, we let $X(W)$ be the word in the matrices $\{D, E,A_1,\ldots, A_{k-1}\}$ that corresponds to the word $W$ in the letters $\{d, e, a_1,\ldots, a_{k-1}\}$.)\footnote{In Equation \eqref{k-ansatz}, the constant $\lambda=\alpha\beta$ is used to slightly generalize the Matrix Ansatz of Theorem \ref{ansatz3} in the same manner that Theorem \ref{ansatz2} generalizes Theorem \ref{ansatz}. The statement of the theorem and the proof are very similar to that of Theorem \ref{ansatz2}, so we do not provide them here.} Similarly to Section \ref{combinatorial_sec}, we show that these matrices give a combinatorial interpretation to the construction of the $k$-RAT. Therefore, the fillings with $\alpha$'s, $\beta$'s, and $q$'s of the maximal tilings of the $k$-rhombic diagrams provide the steady state probabilities for the $k$-PASEP.

In these matrices, the rows are indexed by the tuple $(i,j_1,\ldots,j_{k-1})$ where $i$ is the number of free $d$-strips in a tableau $F$ of the maximal tiling of $\Gamma(W)$ and $j_i$ is the number of $a_i$'s in $W$. The columns of the matrices are indexed by the pair $(i',j_1',\ldots,j_{k-1})$, where $k$ is the number of free $d$-strips in a tableau $F'$ of the maximal tiling of $\Gamma(W d)$ (and respectively, $\Gamma(W e)$ and $\Gamma(W a_s)$ for each $s$)  and $j_i'$ is the number of $a_i$'s in $W d$ (and respectively, $W e$ and $W a_s$ for each $s$). 

Analogously to the construction of the matrices in the two-species PASEP case, we have now

\[
D_{(i,j_1,\ldots,j_{k-1})(i+1,j_1,\ldots,j_{k-1})} = \frac{1}{\beta} \]
and 0 for all other indices.

\[
A_{(i,j_1,\ldots,j_i,\ldots,j_{k-1})(u,j_1,\ldots,j_i+1,\ldots,j_{k-1})} = {i \choose u} q^u \beta^{i-u} \prod_{s=i+1}^{k-1}q^{j_s}
\]
for $0 \leq u \leq i$ and 0 for all other indices.

\[
E_{(i,j_1,\ldots,j_{k-1})(u,j_1,\ldots,j_{k-1})} = \frac{\beta^{i-u}}{\alpha} \left[ {i \choose u}q^u ( q^j+ \alpha [j]_q ) + \alpha \sum_{w=0}^{u-1} {i-u+w \choose i-u} q^w \right]
\]
for $0 \leq u \leq i$ and 0 for all other indices, where we define $j=\sum_{s=1}^{k-1}j_s$, and $[j]_q = q^{j-1}+\ldots+1$.

The relations $DE-qED=D+E$, $DA_i-qA_iD=A_i$, and $A_iE-qEA_i=A_i$ are satisfied by the same arguments as in the two-species PASEP case, except with some additional powers of $q$ in the equations. It remains to show that $A_tA_s=qA_sA_t$ for $t>s$. 

First we compute the $(i,j_1,\ldots,j_s,\ldots,j_t,\ldots,j_{k-1})(u,j_1,\ldots,j_s+1,\ldots,j_t+1,\ldots,j_{k-1})$ entry of $A_tA_s$. (The entries of all other indices are automatically zero).
\begin{multline}
(A_tA_s)_{(i,j_1,\ldots,j_s,\ldots,j_t,\ldots,j_{k-1})(u,j_1,\ldots,j_s+1,\ldots,j_t+1,\ldots,j_{k-1})} \\=\sum_{w=u}^i (A_t)_{(i,j_1,\ldots,j_s,\ldots,j_t,\ldots,j_{k-1})(w,j_1,\ldots,j_s,\ldots,j_t+1,\ldots,j_{k-1})}(A_s)_{(w,j_1,\ldots,j_s,\ldots,j_t+1,\ldots,j_{k-1})(u,j_1,\ldots,j_s+1,\ldots,j_t+1,\ldots,j_{k-1})}\\
=\sum_{w=u}^i {i \choose w} q^w \beta^{i-w} \prod_{r=t+1}^{k-1}q^{j_r} \cdot {w \choose u} q^u \beta^{w-u} \cdot q  \prod_{r=s+1}^{k-1}q^{j_r}\\
=q \sum_{w=u}^i {i \choose w} q^{w+u} \beta^{i-u} \prod_{r=t+1}^{k-1}q^{j_r} \cdot \prod_{r=s+1}^{k-1}q^{j_r}
\end{multline}
Similarly for $A_sA_t$,
\begin{multline}
(A_sA_t)_{(i,j_1,\ldots,j_s,\ldots,j_t,\ldots,j_{k-1})(u,j_1,\ldots,j_s+1,\ldots,j_t+1,\ldots,j_{k-1})}\\ 
= \sum_{w=u}^i (A_s)_{(i,j_1,\ldots,j_s,\ldots,j_t,\ldots,j_{k-1})(w,j_1,\ldots,j_s+1,\ldots,j_t,\ldots,j_{k-1})}(A_t)_{(w,j_1,\ldots,j_s+1,\ldots,j_t,\ldots,j_{k-1})(u,j_1,\ldots,j_s+1,\ldots,j_t+1,\ldots,j_{k-1})}\\
=\sum_{w=u}^i {i \choose w} q^w \beta^{i-w} \prod_{r=s+1}^{k-1}q^{j_r} \cdot {w \choose u} q^u \beta^{w-u} \cdot \prod_{r=t+1}^{k-1}q^{j_r}\\
= \sum_{w=u}^i {i \choose w} q^{w+u} \beta^{i-u} \prod_{r=t+1}^{k-1}q^{j_r} \cdot \prod_{r=s+1}^{k-1}q^{j_r}.
\end{multline}
It is clear that $A_tA_s = qA_sA_t$, as desired.





\section{A Markov chain on the RAT that projects to the two-species PASEP}\label{mc_sec}

We restate here the meaning of a Markov chain that projects to another, and describe the RAT as a Markov chain that projects to the two-species PASEP. Such results exist for the alternative tableaux which project to the regular PASEP, which were originally described in terms of permutation tableaux (which are in simple bijection with the alternative tableaux) in \cite{cw_mc}. Our Markov chain has the same flavor as the existing Markov chain defined by Corteel and Williams. The following definition is from \cite[Definition 3.20]{cw_mc}.\footnote{The results in this section could be extended to the $k$-RAT in the natural way, but we omit the details and proof in this paper.} 

\begin{defn}
Let $M$ and $N$ be Markov chains on finite sets $X$ and $Y$, and let $f$ be a surjective map from $X$ to $Y$. We say that $M$ \emph{projects} to $N$ if the following properties hold:
\begin{itemize}
\item If $x_1, x_2 \in X$ with $Prob_M(x_1 \rightarrow x_2) > 0$, then $Prob_M(x_1 \rightarrow x_2) = Prob_N(f(x_1) \rightarrow f(x_2))$.
\item If $y_1$ and $y_2$ are in $Y$ and $Prob_N(y_1 \rightarrow y_2)>0$, then for each $x_1 \in X$ such that $f(x_1) = y_1$, there is a unique $x_2 \in X$ such that $f(x_2) = y_2$ and $Prob_M (x_1 \rightarrow x_2) > 0$; moreover, $Prob_M (x_1 \rightarrow x_2) = Prob_N (y_1 \rightarrow y_2)$.
\end{itemize}
\end{defn}

Furthermore, we have the following Proposition \ref{mc_walk}, which implies Corollary \ref{cor_projection} below.

Let $\Prob_m(x_0 \rightarrow x; t)$ denote the probability that if we start at state $x_0$ at time 0, then we are in state $x$ at time $t$. From the following proposition of \cite{cw_mc}, we obtain that if $M$ projects to $N$, then a walk on the state diagram of $M$ is indistinguishable from a walk on the state diagram of $N$.

\begin{prop}\label{mc_walk}
Suppose that $M$ projects to $N$. Let $x_0 \in X$ and $y_0,y_1 \in Y$ such that $f(x_0)=y_0$. Then 
\[ 
\Prob_N(y_0 \rightarrow y_1) = \sum_{x' \mbox{ s.t. } f(x')=y_1} \Prob_M(x_0 \rightarrow x_1)
\]
\end{prop}

\begin{cor}\label{cor_projection}
Suppose $M$ projects to $N$ via the map $f$. Let $y \in Y$ and let $X' = \{x \in X\ |\ f(x)=y\}$. Then the steady state probability that $N$ is in state $y$ is equal to the steady state probabilities that $M$ is in any of the states $x \in X'$.
\end{cor}

\begin{wrapfigure}[14]{r}{0.4\textwidth}
\centering
\includegraphics[width=0.35\textwidth]{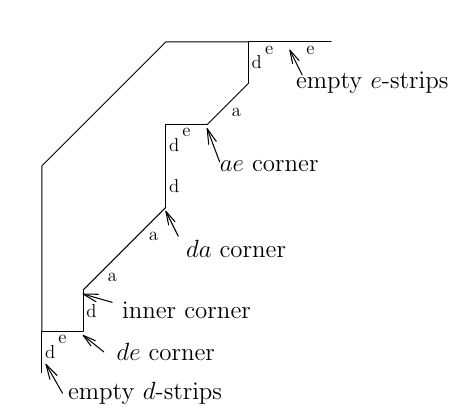}
\caption{The features of a tableau.}
\noindent
\label{RAT_features}
\end{wrapfigure}

In our case, $N$ is the two-species PASEP (which we call the PASEP chain), and $M$ is the Markov chain on the RAT (which we call the RAT chain). 

Recall that $\Omega^n_r$ denotes the states of the two-species PASEP of size $n$ with exactly $r$ ``light'' particles. We specify the states of the RAT chain to be $\Psi_{(n,r)}$, the set of the RAT equivalence classes of size $(n,r)$, based on the fact that different tilings can be chosen to yield equivalent tableaux, as mentioned in Remark \ref{equivalent}. 

Now, we define the transitions on $\Psi_{(n,r)}$ in the RAT chain that correspond to transitions on $\Omega^n_r$ in the PASEP chain. We introduce the following terminology, as in Figure \ref{RAT_features}.

\begin{defn}
A \textbf{corner} is a pair of consecutive $d$ and $e$, $d$ and $a$, or $a$ and $e$-edges on the boundary of a RAT. If there is a $de$ tile, a $da$ tile, or an $ae$ tile (respectively) adjacent to the corresponding edges of the boundary, we call that tile a \textbf{corner tile}. 

An \textbf{inner corner} is a pair of consecutive $e$ and $d$, $a$ and $d$, or $e$ and $a$ edges on the boundary of a RAT.

An \textbf{empty $e$-strip} corresponds to an $e$-edge on the boundary of the RAT that coincides with its top-most boundary.

An \textbf{empty $d$-strip} corresponds to a $d$-edge on the boundary of the RAT that coincides with its left-most boundary.
\end{defn}

\begin{lem}\label{equivalence} Let $\psi \in \Psi_{(n,r)}$ be a RAT equivalence class and let $F \in \psi$. If $F$ has a corner of type $de$, $da$, or $ae$, then there exists an equivalent $F' \in \psi$ that has, respectively, a $de$ tile, a $da$ tile, or an $ae$ tile at that corner.
\end{lem}

\begin{figure}[h]
\centering
\includegraphics[width=0.8\textwidth]{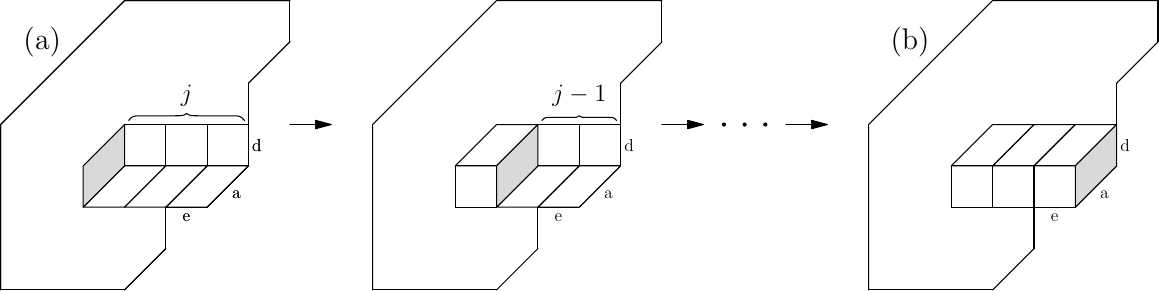}
\caption{If the boundary of the tableau contains consecutively $d$, $a$, and $e$, and there is no $da$ tile adjacent to the $da$ corner, then a ``stack of boxes'' as in (a) must occur in the tiling, for some value of $j$. After performing $j$ flips, the configuration in (b) is obtained, with a $da$ tile adjacent to the $da$ corner, as desired.}
\noindent
\label{DAE_corner}
\end{figure}

\begin{proof}
First, it is clear that any tiling of a rhombic diagram with a $de$ corner must have a $de$ tile at that corner, so for the $de$ case the lemma is obvious.

Now, for the $da$ and the $ae$ cases, it suffices to prove the lemma for only one of them, since by taking the transpose of a tableau and swapping the roles of $\alpha$ and $\beta$, we end up exchanging the $d$'s with the $e$'s (and consequently the $da$ corners with the $ae$ corners), and so by symmetry, these cases will have the same properties. Thus we will prove the $da$ case.

First, if the $da$ corner already has a $da$ tile adjacent to it, we are done. Thus we assume there is not a $da$ tile, which means the tiling of the rhombic diagram must contain the tiles shown in Figure \ref{DAE_corner} (a). More precisely, as seen in the figure, the tiles must be a row of $j \geq 1$ $de$ tiles on top of $j$ $ae$ tiles, with one adjacent $da$ tile on the left. Now it is easy to check that with $j$ flips, we end up with the configuration in Figure \ref{DAE_corner} (b), and moreover, there will a $\beta$ in the corner $da$ tile in the tiling (b) if and only if there is a $\beta$ in the right-most $de$ tile in the tiling (a) (and otherwise there will be a $q$). Thus with $j$ flips, we obtain an equivalent tableau with a $da$ tile in the $da$ corner, as desired.
\end{proof}

Based on the above lemma, we make the following definition:

\begin{defn}\label{corner_flip}
Let $F$ be a tableau with a corner. We call that corner a \textbf{$q$-corner} (or an \textbf{$\alpha$-corner}, or a \textbf{$\beta$-corner}) if a tableau $T$ contains a $q$ in the tile adjacent to that corner (or respectively, an $\alpha$, or a $\beta$) for some $T$ that is equivalent to $F$ and has a corner tile adjacent to the corner.
\end{defn}

\begin{figure}[h]
\centering
\includegraphics[width=0.8\textwidth]{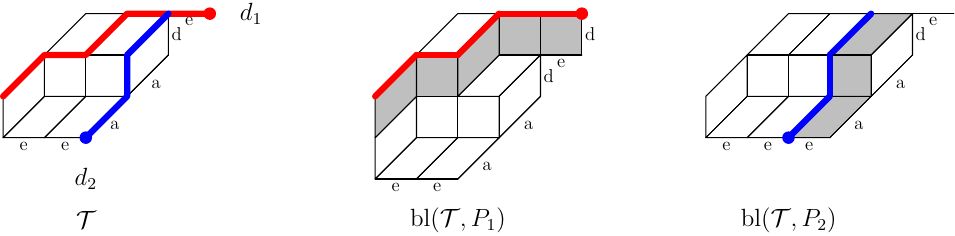}
\caption{Let $d_1$ and $d_2$ be the indicated $d$- and $e$-paths on $\mathcal{T}$. Then $\mathcal{T}$ is the compression of $\bl(\mathcal{T},d_1)$ and $\bl(\mathcal{T},d_2)$ at the highlighted $d$- and $e$-strips, respectively.}
\noindent
\label{compression}
\end{figure}

\begin{defn}\label{west-path}
Let $\mathcal{T}$ be a tiling of a rhombic diagram $F$. A \textbf{$d$-path} on $\mathcal{T}$ is a path from some point on $P_1(F)$ to some point on $P_2(F)$ consisting of A- and $e$-edges. An \textbf{$e$-path} on $\mathcal{T}$ is a path from some point on $P_1(F)$ to some point on $P_2(F)$ consisting of $d$- and $a$-edges. We introduce the operation of \textbf{compressing} a $d$-strip in $\mathcal{T}$ to obtain a new tiling $\mathcal{T}'$ with a $d$-path in place of the $d$-strip (respectively, $e$-strip and $e$-path). We also introduce the inverse operation of \textbf{blowing up} a $d$-path in $\mathcal{T}'$ to obtain a new tiling $\mathcal{T}''$ with a $d$-strip in place of the $d$-path (respectively, $e$-path and $e$-strip). 

Compressing a $d$-strip means selecting its northern border to be the $d$-path, and then gluing together the north and south $e$-edges and $a$-edges of each tile in the $d$-strip, thereby replacing the $d$-strip by the $d$-path. Similarly, compressing an $e$-strip means selecting its western border to be the $e$-path, and then gluing together the west and east $d$-edges and $a$-edges of each tile in the $e$-strip, thereby replacing the $e$-strip with the $e$-path. If $\mathfrak{s}$ is a $d$- or $e$-strip of $\mathcal{T}$, then we denote by $\com(\mathcal{T},\mathfrak{s})$ the new tiling $\mathcal{T}'$ that results from compressing at $\mathfrak{s}$.

For the inverse, blowing up a $d$-path means replacing each $e$-edge of the path with a $de$ tile, and each $a$-edge with a $da$ tile, to obtain a $d$-strip from the new tiles. Similarly, blowing up an $e$-path means replacing each $d$-edge of the path with a $de$ tile, and each $a$-edge with an $ae$ tile, to obtain an $e$-strip from the new tiles. If $\mathfrak{p}$ is a $d$- or $e$-path of $\mathcal{T}$, then we denote by $\bl(\mathcal{T},\mathfrak{p})$ the new tiling $\mathcal{T}''$ that results from blowing up at $\mathfrak{p}$. Figure \ref{compression} illustrates these definitions.

By convention, if $\mathfrak{p}$ is a path of length 0, then blowing up $\mathfrak{p}$ results in replacing it by an empty $e$-strip or an empty $d$-strip (depending on whether $\mathfrak{p}$ coincides with the west boundary or the north boundary of the rhombic diagram, respectively). Conversely, compression of an empty $e$-strip or an empty $d$-strip results in replacing those strips with a single point.
\end{defn}

It is easy to see that compressing is the inverse of blowing up.


Let $F$ be a RAT of size $(n,r,k)$ with tiling $\mathcal{T}$, and let $\psi(F) \in \Psi_{(n,r)}$ denote the equivalence class that $F$ belongs to. Below we describe the RAT chain transitions on $F$, which are also transitions on $\psi(F)$.

\subsection{``Heavy'' particle enters from the left.} If $F$ has an empty $e$-strip $\mathfrak{e}$, then there is a transition in the RAT chain from $F$ that corresponds to a ``heavy'' particle entering from the left in the PASEP. Let the type of $F$ be $e W$.

\begin{figure}[h]
\centering
\includegraphics[width=\textwidth]{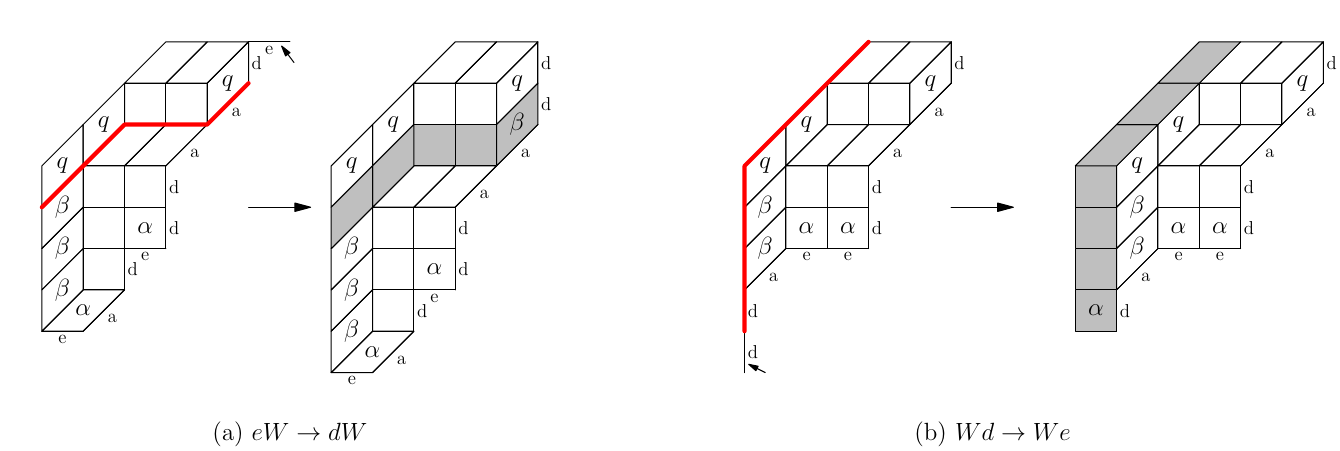}
\caption{For both examples, let the left tableau have tiling $\mathcal{T}$, and denote the indicated empty ($e$- or $d$-) strip by $\mathfrak{e}$ and the marked ($e$- or $d$-) path by $\mathfrak{p}$. Then we obtain a new diagram with tiling $\bl(\com(\mathcal{T},\mathfrak{e}),\mathfrak{p})$, and (a) a $\beta$ is placed in the resulting $d$-strip or (b) an $\alpha$ is placed in the resulting $e$-strip.}
\noindent
\label{mc_legs}
\end{figure}

We define a new RAT $T$ as follows. Let $p$ be the south-most point on $P_1(F)$ (the southeast boundary of $F$) such that there are exactly $n-k-1$ $e$- and $a$- edges on $P_1(F)$ southwest of $p$. Let $\mathfrak{p}$ be any $d$-path originating at $p$. Let $\mathcal{T}' = \bl(\com(\mathcal{T},\mathfrak{e}),\mathfrak{p})$. It is easy to check that $\mathcal{T}'$ is a valid tiling of $\Gamma(d X)$ which has size $(n,r,k+1)$.

If $n-k-1>0$, the new $d$-strip of $\mathcal{T}'$ is non-empty, so we place a $\beta$ in its right-most tile, which is valid since that tile must be either a $de$ tile or a $da$ tile. Furthermore, $p$ was chosen to be the south-most point such that there are $n-k-1$ $e$- and $a$-edges southwest of it, so the right-most tile of the new $d$-strip is also the bottom-most tile of the $a$- or $e$-strip it lies in, and thus does not interfere with the rest of the filling of the tableau. We define $\Prob_{RAT}(F \rightarrow T) = \frac{\alpha}{N+1}$. The weight of $F$ with the exception of $\mathfrak{e}$ equals the weight of $T$ with the exception of the newly added $d$-strip. The weight of the new $d$-strip of $T$ is $\alpha\beta$, and the weight of $\mathfrak{e}$ is $\beta$. Therefore, $\wt(T)=\frac{\alpha\beta}{\beta}\wt(F)$, and so $\wt(F)\Prob_{RAT}(F \rightarrow T) = \frac{\wt(T)}{N+1}$.

For the exceptional case, if $n-k-1=0$, then the newly added $d$-strip of $T$ is empty, and thus has total weight $\alpha$. In this case, the PASEP state corresponding to $F$ is of the form $ed^{n-1}$, and the PASEP state corresponding to $T$ is $d^n$. Then $\wt(F)=\beta\alpha^{n-1}$, $\wt(T)=\alpha^{n-1}$, and so in this case we have $\wt(F)\Prob_{RAT}(F \rightarrow T) = \frac{\beta \wt(T)}{N+1}$.

\subsection{``Heavy'' particle exits from the right.} If $F$ has an empty $d$-strip $\mathfrak{e}$, then there is a transition in the RAT chain from $F$ that corresponds to a ``heavy'' particle exiting from the right in the PASEP. Let the type of $F$ be $W d$.

We define a new RAT $T$ as follows. Let $p$ be the east-most point on $P_1(F)$ such that there are exactly $r+k-1$ $d$- and $a$- edges on $P_1(F)$ northeast of $p$. Let $\mathfrak{p}$ be any $e$-path originating at $p$. Let $\mathcal{T}' = \bl(\com(\mathcal{T},\mathfrak{e}),\mathfrak{p})$. It is easy to check that $\mathcal{T}'$ is a valid tiling of $\Gamma(W e)$ which has size $(n,r,k-1)$.

If $r+k-1>0$, the new $e$-strip of $\mathcal{T}'$ is non-empty, so we place an $\alpha$ in its bottom-most tile, which is valid since that tile must be either a $de$ tile or an $ae$ tile. Furthermore, $p$ was chosen to be the east-most point such that there are $r+k-1$ $d$- and $a$-edges northeast of it, so the bottom-most tile of the new $e$-strip is also the right-most tile of the $a$- or $d$-strip it lies in, and thus does not interfere with the rest of the filling of the tableau. We define $\Prob_{RAT}(F \rightarrow T) = \frac{\beta}{N+1}$. The weight of $F$ with the exception of $\mathfrak{e}$ equals the weight of $T$ with the exception of the newly added $d$-strip. The weight of the new $d$-strip of $T$ is $\alpha\beta$, and the weight of $\mathfrak{e}$ is $\alpha$. Therefore, $\wt(T)=\frac{\alpha\beta}{\alpha}\wt(F)$, and so $\wt(F)\Prob_{RAT}(F \rightarrow T) = \frac{\wt(T)}{N+1}$.

For the exceptional case, if $r+k-1=0$, then the newly added $e$-strip of $T$ is empty, and thus has total weight $\beta$. In this case, the PASEP state corresponding to $F$ is of the form $e^{n-1}d$, and the PASEP state corresponding to $T$ is $e^n$. Then $\wt(F)=\beta\alpha^{n-1}$, $\wt(T)=\alpha^{n-1}$, and so in this case we have $\wt(F)\Prob_{RAT}(F \rightarrow T) = \frac{\alpha \wt(T)}{N+1}$.

\begin{figure}[h]
\centering
\includegraphics[width=\textwidth]{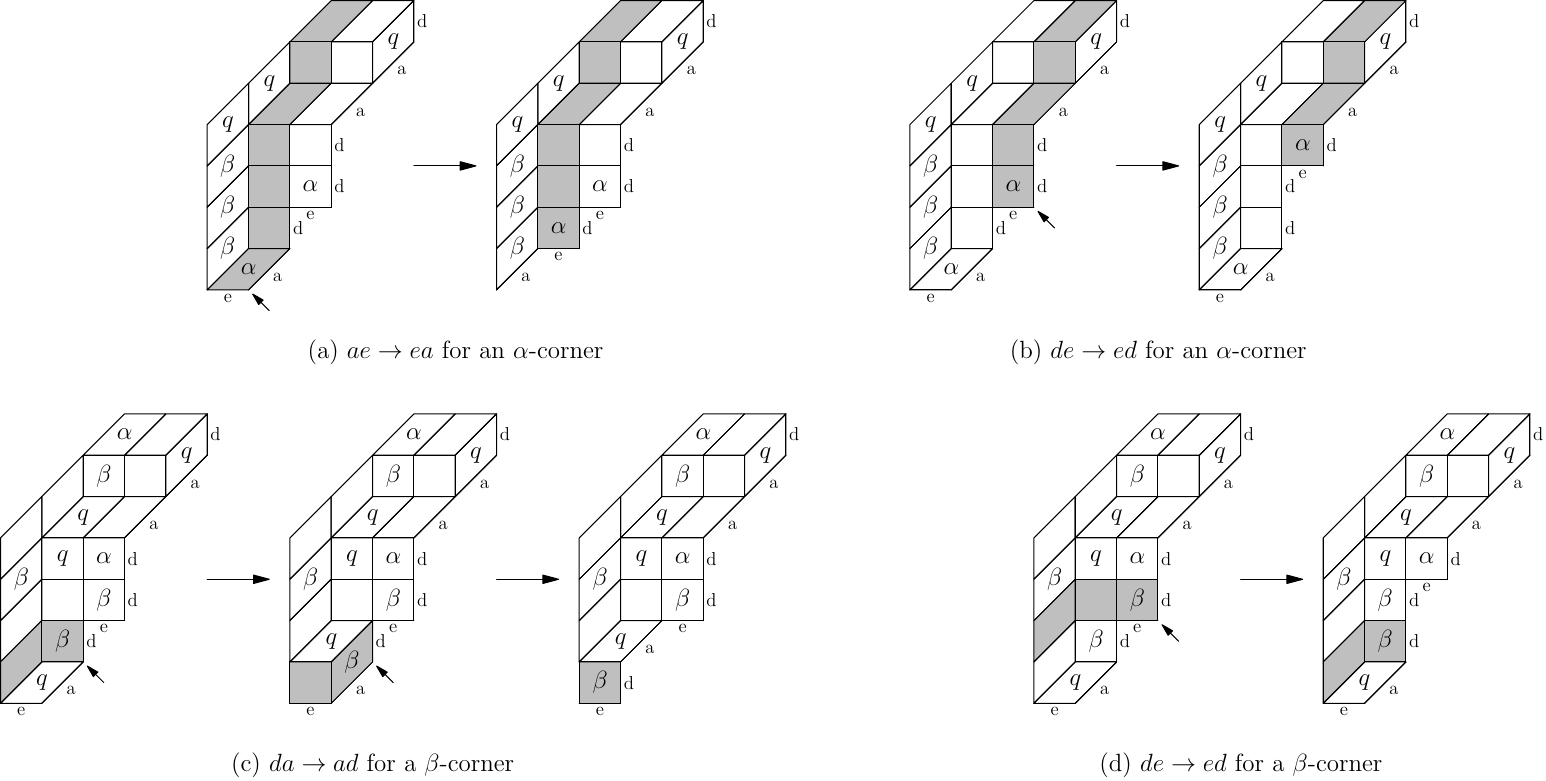}
\caption{(a) and (b) show transitions at an $\alpha$-corner, and (c) and (d) show transitions at a $\beta$-corner.}
\noindent
\label{mc_ab}
\end{figure}

\subsection{A ``heavy'' particle exchanges with a ``hole''.} If $F$ has a $de$ corner, then there is a transition in the RAT chain from $F$ that corresponds to a ``heavy'' particle swapping places with a ``hole'' in the PASEP. Let the type of $F$ be $W de Y$, and suppose it has tiling $\mathcal{T}$. The $de$ corner necessarily corresponds to a $de$ tile. This tile contains an $\alpha$, a $\beta$, or a $q$. We describe these three cases below.

\subsubsection{The $de$ corner tile contains a $\beta$.} We define a new RAT $T$ as follows. Let the $d$-strip containing the $de$ corner tile have length $\lambda$. Let $p$ be the south-most point on $P_1(F)$ such that there are exactly $\lambda-1$ $e$- and $a$- edges on $P_1(F)$ southwest of $p$. Let $\mathfrak{p}$ be any $d$-path originating at $p$. Let $\mathcal{T}' = \bl(\com(\mathcal{T},\mathfrak{e}),\mathfrak{p})$. It is easy to check that $\mathcal{T}'$ is a valid tiling of $\Gamma(W ed Y)$, as in Figure \ref{mc_ab} (d). 

If $\lambda-1>0$, then we place a $\beta$ in the right-most box of the newly inserted $d$-strip $\mathfrak{s}$. Such a filling is valid since the right-most box (containing the new $\beta$) is necessarily the bottom-most box of the $e$- (or $a$-) strip that contains it, and so $\mathfrak{s}$ does not interfere with any of the other tiles in $T$.  We define $\Prob_{RAT}(F \rightarrow T) = \frac{1}{N+1}$. The weight of $T$ equals the weight of $F$. Therefore, $\wt(F)\Prob_{RAT}(F \rightarrow T) = \frac{\wt(T)}{N+1}$.

If $\lambda-1=0$, then necessarily $F$ corresponds to a PASEP state $W ded^j$ for some $j$, and $T$ corresponds to the state $W ed^{j+1}$. The newly added $d$-strip is empty, and so $\wt(F)=\beta\wt(T)$. Therefore, $\wt(F)\Prob_{RAT}(F \rightarrow T) = \frac{\beta\wt(T)}{N+1}$.

\subsubsection{The $de$ corner tile contains an $\alpha$.} We define a new RAT $T$ as follows. Let the $e$-strip containing the $de$ corner tile have length $\lambda$. Let $p$ be the east-most point on $P_1(F)$ such that there are exactly $\lambda-1$ $d$- and $a$- edges on $P_1(F)$ northeast of $p$. Let $\mathfrak{p}$ be any $e$-path originating at $p$. Let $\mathcal{T}' = \bl(\com(\mathcal{T},\mathfrak{e}),\mathfrak{p})$. It is easy to check that $\mathcal{T}'$ is a valid tiling of $\Gamma(W ed Y)$, as in Figure \ref{mc_ab} (b). 

If $\lambda-1>0$, then we place an $\alpha$ in the bottom-most box of the newly inserted $e$-strip $\mathfrak{s}$. Such a filling is valid since the bottom-most box (containing the new $\alpha$) is necessarily the right-most box of the $d$- (or $a$-) strip that contains it, and so $\mathfrak{s}$ does not interfere with any of the other tiles in $T$. We define $\Prob_{RAT}(F \rightarrow T) = \frac{1}{N+1}$. The weight of $T$ equals the weight of $F$. Therefore, $\wt(F)\Prob_{RAT}(F \rightarrow T) = \frac{\wt(T)}{N+1}$.

If $\lambda-1=0$, then necessarily $F$ corresponds to a PASEP state $e^{j}de Y$ for some $j$, and $T$ corresponds to the state $e^{j+1}d Y$. The newly added $d$-strip is empty, and so $\wt(F)=\alpha\wt(T)$. Therefore, $\wt(F)\Prob_{RAT}(F \rightarrow T) = \frac{\alpha\wt(T)}{N+1}$.

\subsubsection{The $de$ corner tile contains a $q$.} We define a new RAT $T$ by simply removing the $de$ corner tile from $F$. We define $\Prob_{RAT}(F \rightarrow T) = \frac{1}{N+1}$. Since a single tile of weight $q$ was removed, $\wt(F) = q \wt(T)$. Therefore, $\wt(F)\Prob_{RAT}(F \rightarrow T) = \frac{q\wt(T)}{N+1}$.

\subsection{A ``heavy'' particle exchanges with a ``light'' particle.} If $F$ has a $da$ corner, then there is a transition in the RAT chain from $F$ that corresponds to a ``heavy'' particle swapping places with a ``light'' particle in the PASEP. Let the type of $F$ be $W da Y$. By Lemma \ref{equivalence}, we can assume that $F$ has a $da$ tile at the $da$ corner. This tile contains a $\beta$ or a $q$. We describe these two cases below.

\subsubsection{The $da$ corner tile contains a $\beta$.} We perform exactly the same operation as for the $de$ case containing a $\beta$. Once again, we define $\Prob_{RAT}(F \rightarrow T) = \frac{1}{N+1}$. In all but the exceptional case, the weight of $T$ equals the weight of $F$. Therefore, $\wt(F)\Prob_{RAT}(F \rightarrow T) = \frac{\wt(T)}{N+1}$.

In the special case, if $F$ corresponds to a PASEP state $W dad^j$ for some $j$, and $T$ corresponds to the state $W ad^{j+1}$, then we have $\wt(F)=\beta\wt(T)$. Therefore, $\wt(F)\Prob_{RAT}(F \rightarrow T) = \frac{\beta\wt(T)}{N+1}$.

\subsubsection{The $da$ corner tile contains a $q$.} We perform exactly the same operation as for the $de$ case containing a $q$. Again, $\wt(F)\Prob_{RAT}(F \rightarrow T) = \frac{q\wt(T)}{N+1}$.

\subsection{A ``light'' particle exchanges with a ``hole''.} If $F$ has an $ae$ corner, then there is a transition in the RAT chain from $F$ that corresponds to a ``light'' particle swapping places with a ``hole'' in the PASEP. Let the type of $F$ be $W ae Y$. By Lemma \ref{equivalence}, we can assume that $F$ has an $ae$ tile at the $da$ corner. This tile contains an $\alpha$ or a $q$. We describe these two cases below.

\subsubsection{The $ae$ corner tile contains an $\alpha$.} We perform exactly the same operation as for the $de$ case containing an $\alpha$. Once again, we define $\Prob_{RAT}(F \rightarrow T) = \frac{1}{N+1}$. In all but the exceptional case, the weight of $T$ equals the weight of $F$. Therefore, $\wt(F)\Prob_{RAT}(F \rightarrow T) = \frac{\wt(T)}{N+1}$.

In the special case, if $F$ corresponds to a PASEP state $e^jae Y$ for some $j$, and $T$ corresponds to the state $e^{j+1}a Y$, then we have $\wt(F)=\alpha\wt(T)$. Therefore, $\wt(F)\Prob_{RAT}(F \rightarrow T) = \frac{\alpha\wt(T)}{N+1}$.

\subsubsection{The $ae$ corner tile contains a $q$.} We perform exactly the same operation as for the $de$ case containing a $q$. Again, $\wt(F)\Prob_{RAT}(F \rightarrow T) = \frac{q\wt(T)}{N+1}$.


\subsection{A lighter particle type exchanges with a heavier particle type.} We describe only the $W ed Y \rightarrow W de Y$ transition, but the same holds true for $W ad Y \rightarrow W da Y$ and $W ea Y \rightarrow W ae Y$ if the corresponding letters are used. If $F$ has an inner $ed$ corner, then there is a transition in the RAT chain from $F$ that corresponds to a ``hole'' swapping places with a ``heavy'' particle in the PASEP. Let the type of $F$ be $W ed Y$. Then to form the tableau $T$, we simply append a $de$ tile to the outside of $F$, adjacent to the $ed$ inner corner. We place a $q$ inside the tile, and thus obtain a valid filling $T$ of type $W de Y$ with a $q$ in its $de$ corner. 

We define $\Prob_{RAT}(F \rightarrow T) = \frac{q}{N+1}$. Therefore, since $q\wt(F) =\wt(T)$, we have  $\wt(F)\Prob_{RAT}(F \rightarrow T) = \frac{\wt(T)}{N+1}$.





The operator $\pr$ is clearly a surjective map from the set $\Psi_{(n,r)}$ to $\Omega^n_r$. It is easy to see by our description of the transitions on the RAT chain that it indeed projects to the PASEP chain.

\begin{figure}[h]
\centering
\includegraphics[width=0.8\textwidth]{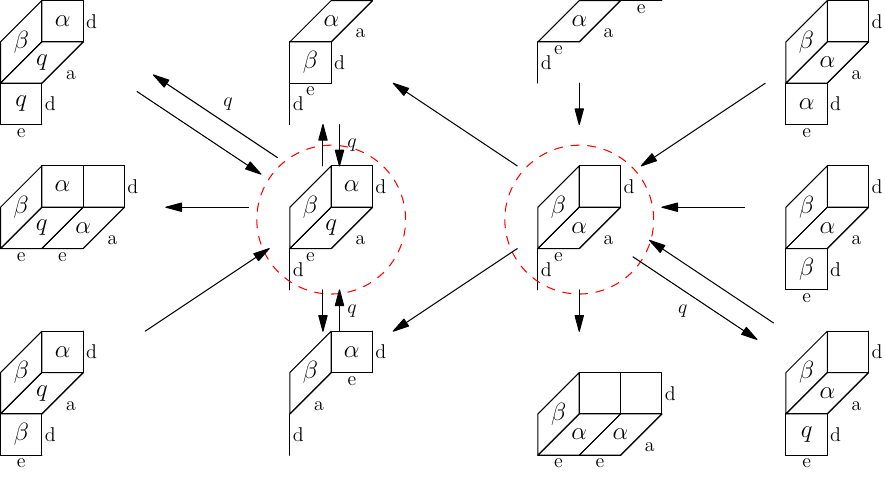}
\caption{Some of the transitions on some of the states in $\Omega^4_1$. All the transitions involving the circled tableaux are included.}
\noindent
\label{mc_example}
\end{figure}

\subsection{Stationary probabilities of the RAT chain}

We carefully summarize the transitions out of a RAT $F$ (and consequently from the equivalence class of $F$), depending on the chosen corner at which the transition occurs. We will be referring to these cases further on. First we make the following definitions. Let $F$ have size $(n,r,k)$ and let $\lambda=(\lambda_1,\ldots,\lambda_{k+r})$ be the partition given by the lengths of the $d$-strips from top to bottom. Assume that $\lambda$ has at least one non-zero part.
\begin{defn} 
We define $\lambda_R$ be the indicator that equals 1 if $F$ has an empty $e$-strip, and 0 otherwise. We define $\lambda_L$ be the indicator that equals 1 if $F$ has n empty $d$-strip, and 0 otherwise.
\end{defn}
\begin{defn}
We call a \emph{$q$-corner} a corner that contains a $q$. (Refer to Definition \ref{corner_flip} for the precise definition.)
We call a \emph{top-most corner} an $\alpha$- or $\beta$-corner such that the length of the $d$-strip containing it equals $\lambda_1$. (If the corner in the top-most position contains a $q$, we do not call it a top-most corner). We define the indicator $\delta^R_{\beta}$ which equals 1 if the top-most corner contains a $\beta$, and 0 if it contains an $\alpha$. Analogously, we call a \emph{bottom-most corner} an $\alpha$- or $\beta$-corner such that the length of the row containing it equals the length of the smallest non-zero row of $\lambda$. (If the corner in the bottom-most position contains a $q$, we do not call it a bottom-most corner).  We define the indicator $\delta^L_{\alpha}$ which equals 1 if the bottom-most corner contains an $\alpha$, and 0 if it contains a $\beta$. We call a \emph{middle corner} an $\alpha$- or $\beta$-corner that is neither a top-most corner or a bottom-most corner (and not a $q$-corner).
\end{defn}

\subsubsection{Summary of transitions $F \rightarrow T$}\label{summary}

Denote by $\pi(F \rightarrow T)$ the rate of transition from tableau $F$ to $T$ (where by rate we mean the unnormalized probability). We obtain the following cases for the transitions from $F$ to $T$.
\begin{enumerate}
\item\label{mid} For a transition at a middle corner, a top-most corner with $\delta^R_{\beta}=1$, or a bottom-most corner with $\delta^L_{\alpha}=1$, we have $\wt(T) = \wt(F)$, and $\pi(F \rightarrow F)=1$.

\item\label{top0} For a transition at a top-most corner with $\delta^R_{\beta}=0$ such that the length of the $e$-strip containing it is greater than 1, we have $\wt(T)=\wt(F)$ and $\pi(F \rightarrow T)=1$. Then the top-most corner of $T$ will be an $\alpha$-corner.
\item\label{bottom0} For a transition at a bottom-most corner with $\delta^L_{\alpha}=0$ such that the length of the row containing it is greater than 1, we have $\wt(T)=\wt(F)$ and $\pi(F \rightarrow T)=1$. Then the bottom-most corner of $T$ will be a $\beta$-corner.

\item\label{top1} For a transition at a top-most corner with $\delta^R_{\beta}=0$ such that the length of the $e$-strip containing it is 1, we have $\wt(T)=\frac{1}{\alpha}\wt(F)$ and $\pi(F \rightarrow T)=1$.
\item\label{bottom1} For a transition at a bottom-most corner with $\delta^L_{\alpha}=0$ such that the length of the $d$-strip containing it is 1, we have $\wt(T)=\frac{1}{\beta}\wt(F)$ and $\pi(F \rightarrow T)=1$.

\item\label{right} For a transition at an empty $e$-strip, we have $\wt(T) = \alpha\wt(F)$ and $\pi(F \rightarrow T)=\alpha$. $T$ will not have an empty $e$-strip, and it will have a top-most corner that contains a $\beta$.
\item\label{left} For a transition at an empty $d$-strip, we have $\wt(T) = \beta\wt(F)$ and $\pi(F \rightarrow T)=\beta$. $T$ will not have an empty $d$-strip, and it will have a bottom-most corner that contains an $\alpha$.

\item\label{inner} For a transition at an inner corner, we have $\wt(T) = q\wt(F)$ and $\pi(F \rightarrow T)=q$. 
\item\label{q-trans} For a transition at a $q$-corner, we have $\wt(T) = \frac{1}{q}\wt(F)$ and $\pi(F \rightarrow T)=1$.
\end{enumerate}

Our main theorem is the following.

\begin{thm}
Consider the RAT chain on $\Psi_{(n,r)}$, the RAT equivalence classes of size $(n,r)$. Fix a RAT $F$ and its equivalence class $\psi$. Then the steady state probability of state $\psi$ is proportional to $\wt(F)$.
\end{thm}

\begin{proof}
To prove the theorem, it suffices to show that for each RAT $F$, the following detailed balance condition holds. Let $\mathcal{R}$ be the set of RAT such that there exists a transition from $F$ to $T \in \mathcal{R}$. Let $\mathcal{S}$ be the set of \emph{equivalence classes} of RAT such that for each $\psi \in \mathcal{S}$, there exists some $S \in \psi$ such that there is a transition from $S$ to $F$. Though we actually work with the equivalence classes, we write for simplicity $S \in \mathcal{S}$.

\begin{equation}\label{balance}
\wt(F) \sum_{T \in \mathcal{R}} \pi(F \rightarrow T) = \sum_{S \in \mathcal{S}} \wt(S)\pi(S \rightarrow F).
\end{equation}

Let the RAT $F$ have type $W$. First we treat the transitions going out of $F$ to $T \in \mathcal{R}$. By the construction of the RAT chain, it is clear that there is a transition with probability 1 for every corner (including the top-most-, bottom-most-, middle-, and $q$-corners), a transition with probability $\alpha$ for an empty $e$-strip, a transition with probability $\beta$ for an empty $d$-strip, and a transition with probability $q$ for every inner corner. These transitions directly correspond to all the possible transitions out of the two-species PASEP state $W$. Suppose $F$ has $C_0$ $q$-corners, $C$ $\alpha$- or $\beta$-corners, and $I$ inner corners. Thus we obtain
\begin{equation}\label{LHS}
\sum_{T \in \mathcal{R}} \pi(F \rightarrow T) = C + C_0 + qI + \alpha \delta_L + \beta \delta_R.
\end{equation}

For the transitions going into $F$ from some $S \in \mathcal{S}$, we observe that any transition from one tableau to another ends with a $q$-corner or an $\alpha$- or $\beta$-corner, an empty $e$-strip, an empty $d$-strip, or an inner corner. Thus it is sufficient to examine all such properties of $F$ to enumerate all the possibilities for $S \in \mathcal{S}$.  We examine the pre-image of the cases for the possible transitions going into $F$ to obtain the following cases for $S$.

\begin{enumerate}
\item For a middle corner, a top-most corner with $\delta^R_{\beta}=0$, or a bottom-most corner with $\delta^L_{\alpha}=0$, we have $\wt(S)=\wt(F)$ and $\pi(S \rightarrow F)=1$. This is the inverse of Case \ref{mid} of Section \ref{summary}. This gives a contribution of $\wt(F)(C-2+(1-\delta^R_{\beta})+(1-\delta^L_{\alpha}))$ to the right hand side (RHS) of the detailed balance equation.\footnote{Note that if $C<2$, the formulas we give have some degeneracies. However, it is easy to verify that these do not cause any problems due to cancellation of all the degenerate terms.}
\item For a top-most corner with $\delta^R_{\beta}=1$ and $\delta_R=0$, we have a transition involving an empty $e$-strip of $S$, so $\wt(S)=\frac{1}{\alpha}\wt(F)$ and $\pi(S \rightarrow F)=\alpha$. This is the inverse of Case \ref{top0} of Section \ref{summary}. This gives a contribution of $\alpha\frac{1}{\alpha}\wt(F)\delta^R_{\beta}(1-\delta_R)$ to the RHS of the detailed balance equation.
\item For a bottom-most corner with $\delta^L_{\alpha}=1$ and $\delta_L=0$, we have a transition involving an empty $d$-strip of $S$, so $\wt(S)=\frac{1}{\beta}\wt(F)$ and $\pi(S \rightarrow F)=\beta$. This is the inverse of Case \ref{bottom0} of Section \ref{summary}. This gives a contribution of $\beta\frac{1}{\beta}\wt(F)\delta^L_{\alpha}(1-\delta_L)$ to the RHS of the detailed balance equation.

\item For a top-most corner with $\delta^R_{\beta}=1$ and $\delta_R=1$, there are two possibilities. For the first, $S$ could fall into Case \ref{top0} of Section \ref{summary}, meaning that the top-most corner of $S$ is a $\beta$-corner, which results in the usual transition with $\wt(S)=\wt(F)$. For the second possibility, $S$ could fall into Case \ref{top1} of Section \ref{summary}, meaning that the top-most corner of $S$ is an $\alpha$-corner and the column containing it has length 1. In that case, $\wt(S)=\alpha\wt(F)$. In both situations, $\pi(S \rightarrow F)=1$. We obtain a contribution of $\wt(F)\delta^R_{\beta}\left(\delta_R +\alpha(1-\delta_R)\right)$  to the RHS of the detailed balance equation.

\item For a bottom-most corner with $\delta^L_{\alpha}=1$ and $\delta_L=1$, there are two possibilities. For the first, $S$ could fall into Case \ref{bottom0} of Section \ref{summary}, meaning that the bottom-most corner of $S$ is an $\alpha$-corner, which is the usual transition with $\wt(S)=\wt(F)$. For the second possibility, $S$ could fall into Case \ref{bottom1} of Section \ref{summary}, meaning that $S$ has a bottom-most corner containing a $\beta$ and the row containing it has length 1. In that case, $\wt(S)=\beta\wt(F)$. In both situations, $\pi(S \rightarrow F)=1$.  We obtain a contribution of $\wt(F)\delta^L_{\alpha}\left(\delta_L +\beta(1-\delta_L)\right)$ to the RHS of the detailed balance equation.

\item For a $q$-corner, we have $\wt(S)=\frac{1}{q}\wt(F)$ and $\pi(S \rightarrow F)=q$. This is the inverse of Case \ref{q-trans} of Section \ref{summary}. We obtain a contribution of $\wt(F)$ to the RHS of the detailed balance equation.

\item For an inner corner, we have $\wt(S)=q\wt(F)$ and $\pi(S \rightarrow F)=1$. This is the inverse of Case \ref{inner} of Section \ref{summary}. We obtain a contribution of $q\wt(F)$ to the RHS of the detailed balance equation.

\end{enumerate}

We sum up the contributions to the RHS of the detailed balance equation to obtain
\begin{multline}\label{RHS}
\sum_{S \in \mathcal{S}} \wt(S)\pi(S \rightarrow F) = \wt(F) ( C +C_0 +qI - \delta^R_{\beta} - \delta^L_{\alpha} + \delta^R_{\beta}(1-\delta_R) + \delta^L_{\alpha}(1-\delta_L) \\
+ \delta^R_{\beta}(\delta_R +\alpha(1-\delta_R)) +\delta^L_{\alpha}(\delta_L +\beta(1-\delta_L)) ).
\end{multline}
We see that after simplification, Equation \ref{RHS} equals Equation \ref{LHS}, so indeed the desired Equation \ref{balance} holds for ``most'' $F$, save for the easily-verified degenerate cases. 

\end{proof}

The proof above circumvents the use of the Matrix Ansatz, and is another way to prove our main result of Theorem \ref{prob}.

As an example, we show some of the transitions on the RAT chain for tableaux of size $(4,1)$ in Figure \ref{mc_example}.

\end{document}